\documentclass[12pt,english,times]{amsart}
\usepackage{amsmath,amssymb,amscd,amsfonts}
\usepackage{amsmath,amscd}
\usepackage{amssymb}
\usepackage{amsthm}
\usepackage{yfonts}
\usepackage{mathrsfs}
\usepackage{MnSymbol}
\usepackage{pictexwd,dcpic}
\usepackage{xfrac}
\usepackage{babel}
\usepackage{enumerate}
\usepackage{hyperref}
\usepackage{anysize}
\usepackage{mdwlist} 
\usepackage{comment}

\usepackage{xcolor}

\newtheorem{thm}{Theorem}[section]
\newtheorem{cor}[thm]{Corollary}
\newtheorem{lem}[thm]{Lemma}
\newtheorem{prop}[thm]{Proposition}

\theoremstyle{definition}
\newtheorem*{defn}{Definition}
\newtheorem{conj}{Conjecture}[]
\newtheorem{ex}{Example}[section]

\theoremstyle{remark}
\newtheorem{remark}[thm]{Remark}

\numberwithin{equation}{section}

\title{Solutions of equations involving the modular $j$ function}

\author[Sebastian Eterovi\'c]{Sebastian Eterovi\'c}
\address{Department of Mathematics, University of California, Berkeley, CA, USA.}
\email{eterovic@math.berkeley.edu}

\author[Sebasti\'an Herrero]{Sebasti\'an Herrero}
\address{Instituto de Matem\'aticas,
Pontificia Universidad Cat\'olica de Valpara\'iso,
Chile.}
\email{sebastian.herrero.m@gmail.com}

\date{\today}

\makeindex

\begin{document}

\maketitle

\begin{abstract}
    Inspired by work done for systems of polynomial exponential equations, we study systems of equations involving the modular $j$ function. We show general cases in which these systems have solutions, and then we look at certain situations in which the modular Schanuel conjecture implies that these systems have generic solutions. An unconditional result in this direction is proven for certain polynomial equations on $j$ with algebraic coefficients.
\end{abstract}


\section{Introduction}

A significant body of work has been produced towards studying systems of  polynomial exponential equations, and in particular, to determine which algebraic varieties $V\subseteq\mathbb{C}^{2n}$ have generic points of the form $(x_{1},\ldots,x_{n},\exp(x_{1}),\ldots,\exp(x_{n}))$, where $\exp$ denotes the usual exponential function $\exp(x)=e^x$ for $x$ in $\mathbb{C}$. Similarly, it is of interest to determine which algebraic varieties $V\subseteq\mathbb{C}^{n+1}$ have generic points of the form $(x,\exp(x),\exp \circ \exp(x),\ldots,\exp \circ\, \cdots \circ \exp(x))$; see \cite{bays-kirby}, \cite{daquino2}, \cite{kirby-zilber}, \cite{mantova}, \cite{marker} for some important results in this area. These questions are in great part motivated by the work of  Zilber on pseudo-exponentiation (see \cite{kirby-zilber} and \cite{zilberexp}), but due to their geometric nature, they still make sense if we replace $\exp$ by another holomorphic function. In this paper we focus our attention on the modular $j$ function, which is the unique holomorphic function defined on the upper-half plane $\mathbb{H}:=\{z\in \mathbb{C}:\mathrm{Im}(z)>0\}$ that is invariant under the action of the modular group $\mathrm{SL}_2(\mathbb{Z})$ and has a Fourier expansion of the form
\begin{equation}\label{eq:j-fourier-expansion}
j(z)=q^{-1}+\sum_{k=0}^{\infty}a_kq^k \text{ with }q:=\exp(2\pi i z)  \text{ and }a_k\in \mathbb{C} 
\end{equation}
(see \textsection\ref{sec:background} for details). There are deep reasons for choosing to work with this function. On one hand,  there are  various results showing that the $j$ function has algebraic and transcendence properties that parallel those of $\exp$.  
For example, we know how to use specific values of $j$ and $\exp$ to construct abelian extensions of imaginary quadratic fields and abelian extensions of $\mathbb{Q}$, respectively; Schneider's theorem \cite{schneider} gives an analogue of Lindemann's theorem\footnote{Lindemann's theorem states that $\mathrm{tr.deg.}_{\mathbb{Q}}(\mathbb{Q}(z,\exp(z))) = 0$ if and only if $z=0$.} for $\exp$;
and more recently, Pila and Tsimerman \cite{pila-tsimerman} have proven a functional transcendence property for the $j$ function which is an analogue of the Ax--Schanuel theorem \cite[Corollary 1]{ax} for $\exp$. On the other hand, the main theorem  in \cite{B-SDGP} states that $j(z)$ is transcendental if $\exp(2\pi i z)$ is algebraic, giving a concrete connection between the algebraic and transcendence properties of these functions. 


The purpose of  this paper is to obtain analogues for the  $j$ function of some of the  results for $\exp$ proven by D'Aquino, Fornasiero and Terzo \cite{daquino2}, Mantova \cite{mantova} and Marker \cite{marker}.  
%
Specifically, the motivating questions of this paper are the following:
\begin{enumerate}
    \item[1.] Under what conditions on a given irreducible algebraic variety $V\subseteq \mathbb{C}^{2n}$ can we ensure that $V$ contains a point of the form $(z_{1},\ldots,z_{n},j(z_{1}),\ldots,j(z_{n}))$ with $z_1,\ldots,z_n\in \mathbb{H}$, and furthermore, that there is such a point which is generic over a given finitely generated subfield of $\mathbb{C}$?
     \item[2.] Under what conditions  on a given irreducible algebraic variety $V\subseteq \mathbb{C}^{n+1}$ can we ensure that $V$ contains a point of the form $(z,j_1(z),\ldots,j_n(z))$, where $j_n$ denotes the $n$-th fold composition of $j$ with itself and $z$ is in the domain of definition of $j_n$, and when can we assure that generic (over a given finitely generated subfield of $\mathbb{C}$) points exist?
\end{enumerate}

Some versions of the first question  have been studied by Aslanyan in the setting of differential fields; see \cite{vahagn2} and \cite{vahagn}.

\subsection{Main results}
The first two main theorems of this paper give partial answers to the first question. More specifically, Theorem \ref{th:main1} bellow 
gives the existence of points of the form $(z_1,\ldots,z_n,j(z_1),\ldots,j(z_n))$ for a certain family of varieties, and Theorem \ref{th-main-2} provides a conditional result on the existence of generic points of this form in plane curves. In order to state these results, we introduce the following notation. Given a positive integer $n$ we define
$$E_j^n:=\{(z_1,\ldots,z_n,j(z_1),\ldots,j(z_n)):z_1,\ldots,z_n\in \mathbb{H}\},$$
which is a subset of $\mathbb{H}^n\times \mathbb{C}^{n}$.

\begin{thm}\label{th:main1}
Let $V\subseteq \mathbb{C}^{2n}$ be an irreducible algebraic variety and let $\pi:\mathbb{C}^{2n}\rightarrow\mathbb{C}^{n}$ be the projection onto the first $n$ coordinates. If $\pi(V)$ is Zariski dense in $\mathbb{C}^{n}$, then $\pi(E_j^n\cap V)$ is Zariski dense in $\mathbb{C}^n$. In particular, $V$ contains infinitely many points of the form $(z_{1},\ldots,z_{n},j(z_{1}),\ldots,j(z_{n}))$.
\end{thm}

One of the key ingredients in our proof of Theorem \ref{th:main1} is Proposition \ref{prop:masser} in \textsection \ref{sec:equations_automorphic}, which is the automorphic analogue of a result due to  Brownawell and Masser on the existence of solutions of certain systems of  polynomial exponential equations (see \cite[Theorem 2.1]{daquino2} and also \cite[Proposition 2]{brownawell-masser}). The version of Theorem \ref{th:main1} for $\exp$ is \cite[Lemma 2.10]{daquino2}.


The second main theorem of this paper shows that the modular Schanuel conjecture (Conjecture \ref{conj:msc} in \textsection\ref{subsec:msc}) implies the existence of generic points in $E_j^1\cap V$ when $V$ is an irreducible plane algebraic curve that is not a horizontal or vertical line. Conjecture \ref{conj:msc} is the modular version of Schanuel's classical conjecture for the exponential function (see \cite[pp.~30--31]{lang2}). We restrict to curves that are not horizontal nor vertical lines since those cases are easy to analyze  (see \textsection \ref{sec:excl}). 

\begin{thm}\label{th-main-2}
Let $V\subset \mathbb{C}^2$ be an irreducible algebraic curve  that is not a horizontal nor a vertical line, and let $K$ be a finitely generated subfield  of $\mathbb{C}$ over which $V$ is defined. Then the modular Schanuel conjecture  implies that  there exist infinitely many points in $V$ of the form $(z,j(z))$ with $z\in \mathbb{H}$ and $\mathrm{tr.deg.}_{K}(K(z,j(z))) = 1$. 
\end{thm}

The corresponding (conditional) result for the exponential function is due to Mantova; see \cite[Theorem 1.2]{mantova}. There, the author reduces the problem, assuming Schanuel's conjecture, to a finiteness statement on rational solutions of certain polynomial exponential equations, which is proven unconditionally in an appendix written together with Zannier. The proof of that statement is  based on Baker's theorem on linear forms in logarithms.

For curves defined over $\overline{\mathbb{Q}}$ (where $\overline{\mathbb{Q}}$ denotes the algebraic closure of $\mathbb{Q}$ in $\mathbb{C}$), we prove an unconditional version of Theorem \ref{th-main-2} in \textsection \ref{subsec:unconditional_result} (see Proposition \ref{prop:unconditional}). 

The final two main results of this paper are about solutions of equations involving compositions of $j$ with itself, and give partial answers to our second question. Even though expressions like $j_2(z)=j(j(z))$ are not defined in all of $\mathbb{H}$, we can still find solutions in some situations. For a  positive integer $n$, we denote by $\mathbb{H}_n$ the maximum domain of definition of $j_n$ (see \S \ref{sec:background} for details). 

\begin{thm}\label{th:main3}
Let $V\subset \mathbb{C}^{n+1}$ be an algebraic hypersurface defined by an irreducible polynomial $p(X,Y_{1},\ldots,Y_{n})$ in $\mathbb{C}[X,Y_{1},\ldots,Y_{n}]$ with $\frac{\partial p}{\partial Y_{n}} \neq 0$. Then there are infinitely many points $z$ in $\mathbb{H}_n$ such that $(z, j(z),\ldots, j_{n}(z))\in V$.
\end{thm}

Assuming the modular Schanuel conjecture, we prove the existence of generic points in the following setting.

\begin{thm}\label{thm:main4}
Let $V\subset \mathbb{C}^3$ be an algebraic variety defined by an irreducible polynomial  $p(X,Y_{1},Y_{2})$ in $\overline{\mathbb{Q}}[X,Y_{1},Y_{2}]$ with $\frac{\partial p}{\partial X},\frac{\partial p}{\partial Y_2}\neq 0$. Then, the modular Schanuel conjecture implies that there exist infinitely many points $z$  in $\mathbb{H}_2$ such that $(z,j(z),j_2(z))\in V$ and $\mathrm{tr.deg.}_{\mathbb{Q}}(\mathbb{Q}(z,j(z),j_{2}(z))) = 2$.
\end{thm}

The corresponding result for the exponential function, conditional on Schanuel's conjecture, is due to D'Aquino, Fornasiero and Terzo; see \cite[Theorem 4.2]{daquino2}.

We point out that most of the proofs  use  only elementary tools from algebraic geometry, complex analysis, the theory of automorphic functions and class field theory (with the exception of Proposition \ref{prop:gsol} in \S \ref{subs:G_orbit_bounded}), 
while at the same time we take inspiration from ideas found in \cite{daquino2} and \cite{mantova}. It  is then  natural  to  wonder  whether  our  methods    can  be extended and applied to other functions; see e.g.~\cite{pila13} and~\cite{pila-tsimerman} for interesting transcendence results involving derivatives of the $j$ function, $\exp$ and elliptic functions. It would be particularly interesting to know whether these ideas could provide new proofs of known results for $\exp$. Motivated by this,  we show in \S \ref{sec:further} two examples of how one can use the methods of this paper to find solutions to certain equations involving  the functions  $j'$ (the usual derivative of $j$) and $\exp(1/z)$, respectively. At the same time, we believe it is possible to obtain stronger results than the ones presented in this paper with the aid of more sophisticated tools (such as o-minimality).  We plan to pursue the study of this type of problems in more detail in future work. 


\subsection{Structure of the paper}

In \textsection\ref{sec:background}, we introduce some extra notation and provide some background material that is used in the next sections. The modular Schanuel conjecture can be found in \textsection\ref{subsec:msc}. 

In \textsection\ref{sec:excl} we give a couple of simple examples of varieties where the problems we are interested in are easy to study. 

The purpose of \textsection\ref{sec:equations_automorphic} is to prove Proposition \ref{prop:masser}, which shows that certain systems of analytic equations involving meromorphic automorphic functions have solutions. This proposition plays a crucial role in the proofs of all of our main theorems. 

In \textsection\ref{sec:triangular_forms} we introduce and study certain family of affine varieties that have a very simple form. These are used in our proofs of Theorems \ref{th:main1} and \ref{th-main-2}, and also in the proof of Proposition \ref{prop:unconditional} in \textsection\ref{subsec:unconditional_result}. 

The proofs of our main theorems are contained in \textsection\ref{sec:proof_of_thm_1}, \textsection\ref{sec:proof_of_thm_2}, \textsection\ref{sec:proof_of_thm_3} and \textsection\ref{sec:proof_of_thm_4}, following the order of their presentation in the introduction. 

Finally,  in \textsection\ref{sec:further} we adapt the ideas used in  the proof of Proposition \ref{prop:masser} to prove the existence of solutions to certain equations involving  $j'$ and $\exp(1/z)$.


\section{Background material and notation}
\label{sec:background}

Given sets $A,B$ we define $A\setminus B:=\{a\in A:a\not \in B\}$. 
Following the notation used in the introduction, we denote by $\mathbb{H}$ the complex upper-half plane $\{z\in \mathbb{C}:\mathrm{Im}(z)>0\}$. The group $\mathrm{GL}_2^+(\mathbb{R})$ of $2$ by $2$ matrices with coefficients in $\mathbb{R}$ and positive determinant acts on $\mathbb{H}$ via the formula $$gz:=\frac{az+b}{cz+d} \ \text{ for } \ g=\left(\begin{array}{cc}a & b\\ c & d\end{array}\right) \text{ in }\mathrm{GL}^+_{2}(\mathbb{R}).$$ This action can be extended to a continuous action of $\mathrm{GL}_2^+(\mathbb{R})$ on the Riemann sphere $\widehat{\mathbb{C}}:=\mathbb{C}\cup \{\infty\}$. Given a subring $R$ of $ \mathbb{R}$ we define $M_2^+(R)$ as the set of $2$ by $2$ matrices with positive determinant and coefficients in $R$. We put $$G:=\mathrm{GL}_2^+(\mathbb{Q})= M_2^+(\mathbb{Q}),$$ which is a subgroup of $\mathrm{GL}_2^+(\mathbb{R})$. The modular group is defined as $$\Gamma:=\mathrm{SL}_2(\mathbb{Z})=\{g\in M_2^+(\mathbb{Z}):\det(g)=1\}.$$ The modular $j$ function was defined in the introduction as the unique holomorphic function $j:\mathbb{H}\to \mathbb{C}$ that satisfies $$j(gz)=j(z) \text{ for every }g \text{ in } \Gamma \text{ and every }z\text{ in }\mathbb{H},$$ and has a Fourier expansion of the form \eqref{eq:j-fourier-expansion}. It induces an analytic isomorphism of Riemann surfaces $\Gamma \backslash \mathbb{H}\simeq \mathbb{C}$. The quotient space $Y_{\Gamma}=\Gamma\backslash \mathbb{H}$ is known to be a moduli space for complex tori, or equivalently, elliptic curves over $\mathbb{C}$. If $\Gamma z$ is a point in $Y_{\Gamma}$ and $E_z$ denotes an elliptic curve in the corresponding isomorphism class, then $j(z)$ is simply the $j$-invariant of the curve $E_z$. 

Given a point $\mathbf{z}=(z_1,\ldots,z_n)$ in $\mathbb{H}^n$, we simply write $j(\mathbf{z})$ instead of $(j(z_1),\ldots,j(z_n))$. For a positive integer $n$, we define $j_n$ inductively by $$j_1=j \text{ and }j_{n+1}=j\circ j_n \text{ for }n\geq 1.$$ The domain of definition of $j_n$, denoted by $\mathbb{H}_n$, is also defined inductively by $$\mathbb{H}_{1}=\mathbb{H} \text{ and }\mathbb{H}_{n+1}=\{z\in \mathbb{H}_n:j(z)\in \mathbb{H}\} \text{ for }n\geq 1.$$

In this paper, by an open subset of $\mathbb{C}^n$ we mean an open subset with respect to the Euclidean (i.e.~analytic) topology. This will be emphasized when needed. Given a field $K$, terms like Zariski open, Zariski closure and Zariski dense, referring to subsets of $K^n$, are meant to be understood as defined with respect to the Zariski topology. 

\subsection{Modular polynomials}\label{sec:modular_polynomials} Let $\left\{\Phi_{N}(X,Y)\right\}_{N=1}^{\infty}\subseteq\mathbb{Z}[X,Y]$ denote the family of \emph{modular polynomials} associated to $j$ (see \cite[Chapter 5, Section 2]{lang} for the definition and main properties of this family). We recall that $\Phi_{N}(X,Y)$ is irreducible in $\mathbb{C}[X,Y]$, $\Phi_{1}(X,Y) = X-Y$, and for $N\geq 2$, $\Phi_{N}(X,Y)$ is symmetric of total degree $\geq 2N$. Also, the action of $G$ on $\mathbb{H}$ can be traced by using modular polynomials in the following way: for every $g$ in $G$ we define $\mathrm{red}(g)$ as the unique matrix of the form $rg$ with $r\in \mathbb{Q},r>0$ such that the entries of $rg$ are all integers and relatively prime. Then, for every $z_1,z_2$ in $\mathbb{H}$ the following statements are equivalent: 
\begin{itemize}
    \item[(M1):] $\Phi_{N}(j(z_1),j(z_2)) = 0$,
    \item[(M2):] $gz_1=z_2$ for some $g$ in $G$ with $\det\left(\mathrm{red}(g)\right) = N$.
\end{itemize}

\subsection{Special and ordinary points}

A point $z$ in $\mathbb{H}$ is said to be \emph{special} if there is a matrix $g$ in $G$ such that $z$ is the unique fixed point of $g$ in $\mathbb{H}$. This is equivalent to saying that $z$ satisfies a non trivial quadratic equation with integer coefficients. A theorem of Schneider (\cite{schneider}) says that $\mathrm{tr.deg.}_{\mathbb{Q}}(\mathbb{Q}(z,j(z))) = 0$ if and only if $z$ is special. The special points of $\mathbb{H}$ are exactly those points for which the corresponding elliptic curve (more precisely, any representative in the corresponding isomorphism class of elliptic curves) has complex multiplication.  For this reason, special points are also known as \emph{CM points} in the literature.

Special points are deeply linked to class field theory for imaginary quadratic fields. The following result is a well known application of that relation. It also uses a classical result due to Hecke, Deuring, Mordell and Heilbronn stating that for every positive integer $h$ there are only finitely many maximal quadratic imaginary orders of class number $h$. This result extends to non-maximal quadratic imaginary orders by \cite[Theorem 7 in Chapter 8, \S 1]{lang}. 


\begin{lem}\label{lem:class_number}
Let $M$ be a positive integer. Then, the set of $\Gamma$-orbits of special points $z$ in $\mathbb{H}$ for which the degree $[\mathbb{Q}(z,j(z)):\mathbb{Q}(z)]$ is bounded above by $M$ is finite.
\end{lem}
\begin{proof}
If $z$ is a special point in $\mathbb{H}$ and $E_z$ denotes an elliptic curve in the corresponding isomorphism class, then the ring of endomorphisms of $E_z$ defined over $\mathbb{C}$ is isomorphic to an imaginary quadratic order $\mathcal{O}_z$. It is known that 
the field $\mathbb{Q}(z,j(z))$ is the ring class field of $\mathcal{O}_z$, and the correspondence $z\mapsto \mathcal{O}_z$ induces a finite to one surjective map between the collection of all $\Gamma$-orbits of special points in $\mathbb{H}$ and the collection of all (isomorphism classes of) imaginary quadratic orders. Under this map, the class number $h(\mathcal{O}_z)$ of $\mathcal{O}_z$ equals $[\mathbb{Q}(z,j(z)):\mathbb{Q}(z)]$. 
As mentioned above, for every positive integer $h$ there are only finitely many imaginary quadratic orders with class number $h$.  Therefore, if $[\mathbb{Q}(z,j(z)):\mathbb{Q}(z)]$ is bounded above by $M$, then there is a finite set $S_M\subset \mathbb{H}$ of special points, depending only on $M$, such that $z\in \Gamma \cdot S_M=\{g z_0:g\in \Gamma,z_0\in S_M\}$. This proves the lemma.
\end{proof}

We extend the definition of special point to higher dimensions as follows. We say that a point $\mathbf{z}$ in $\mathbb{H}^{n}$ is \emph{special} if every coordinate of $\mathbf{z}$ is special. On the other hand, we say that $\mathbf{z}$ is \emph{ordinary} if no coordinate of $\mathbf{z}$ is special. 



\subsection{The modular Schanuel conjecture}
\label{subsec:msc}
We now state an important conjecture which, just like Schanuel's conjecture for $\exp$, is a special case of the generalised period conjecture of Grothendieck--Andr\'e (see \cite[\textsection 23.4.4]{andre}, \cite[\textsection 1 Conjecture modulaire]{bertolin}, and \cite[Conjecture 8.3]{pila2}). 

\begin{conj}[Modular Schanuel's Conjecture]
\label{conj:msc}
If $z_{1},\ldots,z_{n}$ in $\mathbb{H}$ are non-special points in distinct $G$-orbits, then:
\begin{equation}
    \mathrm{tr.deg.}_{\mathbb{Q}}(\mathbb{Q}\left(z_{1},\ldots,z_{n},j(z_{1}),\ldots,j(z_{n})\right))\geq n.
\end{equation}
\end{conj}

For the rest of the paper we will refer to Conjecture \ref{conj:msc} as MSC for short. The merit of this conjecture relies not only on it coming from the generalised period conjecture, but there are also results like the Ax--Schanuel theorem for $j$ (see \cite[Theorem 1.3]{pila-tsimerman}) saying that an inequality in the spirit of MSC holds on differential fields that have a $j$-function. 

We now give an alternative formulation of MSC that will be used in \textsection\ref{sec:triangular_forms} and \textsection\ref{sec:proof_of_thm_2}. Given a subset  $A$ of $\mathbb{H}$, we define $\dim_{G}(A)$ as the number of distinct $G$-orbits in $$G\cdot A=\{g a:g\in G,a\in A\}$$ (this number can be infinite)\footnote{Note that $\dim_{G}(A)$ is the dimension of the set $A$ with respect to the pregeometry $(\mathbb{H},\mathrm{cl})$ where the closure map is defined as $\mathrm{cl}(A):=G\cdot A$.}. Equivalently, $\dim_G(A)$ is the cardinality of the quotient set $G\backslash (G\cdot A)$. Given another subset $C\subseteq\mathbb{C}$, we define $\dim_{G}(A|C)$ 
as the number of distinct $G$-orbits in $(G\cdot A)\setminus (G\cdot C)$. Let $\Sigma\subset\mathbb{H}$ be the set of all special points. Note that, by Schneider's theorem and the equivalence between (M1) and (M2) in \textsection\ref{sec:modular_polynomials}, MSC is equivalent to the following statement: for any $z_{1},\ldots,z_{n}$ in $\mathbb{H}$ we have 
\begin{equation*}
    \mathrm{tr.deg.}_{\mathbb{Q}}(\mathbb{Q}\left(z_{1},\ldots,z_{n},j(z_{1}),\ldots,j(z_{n})\right))\geq \dim_{G}\left(\{z_{1},\ldots,z_{n}\}|\Sigma\right).
\end{equation*}

\subsection{Generic points}

Given a subfield $K$ of $\mathbb{C}$ and a point $\mathbf{x}$ in $\mathbb{C}^n$, we denote by $K[\mathbf{x}]$ and $K(\mathbf{x})$ the subring and the subfied of $\mathbb{C}$, respectively, generated by $K$ and the coordinates of $\mathbf{x}$. If $\mathbf{x}=(x_1,\ldots,x_n)$ then we simply write $K[\mathbf{x}]=K[x_1,\ldots,x_n]$ and $K(\mathbf{x})=K(x_1,\ldots,x_n)$. We also denote by $\overline{K}$ the algebraic closure of $K$ in $\mathbb{C}$. Given a collection $S$ of polynomials in $n$ variables and complex coefficients, we denote by $V(S)$ the affine subvariety of $\mathbb{C}^n$ defined as the zero locus of the polynomials in $S$. If $S$ is the finite set $\{p_1,\ldots,p_m\}$, then we write $V(S)=V(p_1,\ldots,p_m)$.

Let $V$ be an algebraic subvariety of $\mathbb{C}^n$ of dimension $d$ defined over a subfield $K$ of $\mathbb{C}$. A point $\mathbf{x}$ in $V$ is called \emph{generic over $K$} if $\mathrm{tr.deg.}_{K}(K(\mathbf{x}))=d$. For later use, we now recall two well known results from algebraic geometry. We include proofs for the reader's convenience.

\begin{lem}\label{lem:generic}
Let $K$ be an algebraically closed field and let $L\supseteq K$ be a field extension. Assume that $V$ and $W$ are algebraic subvarieties of $L^n$ defined over $K$ such that $V\cap K^n$ is irreducible and of dimension $d$. Moreover, assume that there exists a point $\mathbf{x}$ in $L^n$ such that $\mathbf{x}\in V\cap W$ and $\mathrm{tr.deg.}_{K}(K(\mathbf{x}))=d$. Then, we have $V \subseteq W$.
\end{lem}
\begin{proof}
Define $I$ as the set of polynomials in $K[X_1,\ldots,X_n]$ vanishing on $\mathbf{x}$. Clearly $I$ is a prime ideal of $K[X_1,\ldots,X_n]$. Let $Z$ denote the corresponding irreducible algebraic set in $K^n$. We have that $K[X_1,\ldots,X_n]/I$ is isomorphic to $K[\mathbf{x}]$. Since the field of fractions of $K[\mathbf{x}]$ has transcendence degree over $K$ equal to $d$, we conclude that $Z$  has dimension $d$ over $K$. By hypothesis, $V$ is defined by certain polynomials $p_1,\ldots,p_m$ in $K[X_1,\ldots,X_n]$. Since $\mathbf{x}\in V$, we have $p_1,\ldots,p_m\in I$. This implies that $Z\subseteq V\cap K^n$. Similarly, $Z\subseteq W\cap K^n$. Since $V\cap K^n$ is irreducible over $K$ and of the same dimension than $Z$, we have $Z=V\cap K^n$ (see, e.g.~\cite[Theorem 1.19]{Shafarevich1}). We conclude that $V\cap K^n\subseteq W\cap K^n$. Since $K$ is algebraically closed, we have that every polynomial in $K[X_1,\ldots,X_n]$ defining $W$ is contained in the radical of the ideal generated by $p_1,\ldots,p_m$. This implies that $V\subseteq W$ and completes the proof of the lemma.
\end{proof} 


\begin{cor}\label{cor:generic_contention}
Let $V$, $W$ be affine varieties in $\mathbb{C}^{n}$ with $V$ irreducible of dimension $d$. Let $\pi:\mathbb{C}^{n}\to \mathbb{C}^d$ be the projection map over a fixed choice of $d$ different coordinates, so that $\pi(x_1,\ldots,x_n)=(x_{i_1},\ldots,x_{i_d})$ with $\{i_1,\ldots,i_d\}$ a subset of cardinality $d$ of $\{1,\ldots,n\}$. If $\pi(V\cap W)$ contains a non-empty Euclidean open subset of $\mathbb{C}^d$, then $V\subseteq W$.
\end{cor}
\begin{proof}
Let $K_0$ be a finitely generated subfield of $\mathbb{C}$ over which $V$ and $W$ are defined. Put $K=\overline{K_0}$ and choose a non-empty Euclidean open subset $U$ of $\mathbb{C}^d$ contained in $\pi(V\cap W)$. Since $\mathrm{tr.deg.}_{K}(\mathbb{C})$ is infinite, we can find $\mathbf{x}$ in $U$ with  $\mathrm{tr.deg.}_{K}(K(\mathbf{x}))=d$. Let $\mathbf{y}$ be any point in $V\cap W$ with $\pi(\mathbf{y})=\mathbf{x}$. We have $$d=\mathrm{tr.deg.}_{K}(K(\mathbf{x}))\leq \mathrm{tr.deg.}_{K}(K(\mathbf{y}))\leq \mathrm{dim}_K(V)=\mathrm{dim}_{\mathbb{C}}(V)=d,$$ hence $\mathbf{y}$ is generic in $V$ over $K$ and by Lemma \ref{lem:generic} we get $V\subseteq W$. This proves the result.
\end{proof}

\section{Some simple examples}
\label{sec:excl}
As explained in the introduction, the first  problem that we are interested in is to find conditions that ensure that if an irreducible algebraic variety $V\subseteq\mathbb{C}^{2n}$ satisfies them, then $E_{j}^{n}\cap V$ is non-empty. 
Here we look at some simple examples of varieties where this problem is easy to analize.

\begin{ex}\label{ex:1} Let $V\subset \mathbb{C}^{2}$ be the vertical line defined by the equation $X = r$, for some $r$ in $\mathbb{C}$. We have $E_j^1\cap V=\emptyset$ if $r\in \mathbb{C}\setminus \mathbb{H}$, while $E_j^1\cap V=\{(r,j(r))\}$ if $r\in \mathbb{H}$. 
More generally, no variety $V\subset\mathbb{C}^{2n}$ contained in a hyperplane defined by an equation of the form $\left\{X_{i} = r\right\}$ for some $i$ in $\{1,\ldots,n\}$ and $r$ in $ \mathbb{C}\setminus \mathbb{H}$ can  intersect $E_{j}^{n}$.
\end{ex}

\begin{ex} Let $V\subset \mathbb{C}^{2}$ be the horizontal line defined by the equation $Y = r$, for some $r\in\mathbb{C}$. Since $j(\mathbb{H})=\mathbb{C}$, we can choose $z_0$ in $\mathbb{H}$ with $j(z_0)=r$. Then, we have $$E_j^1\cap V=\{(g z_0,j(z_0)):g \in \Gamma\}.$$ 
\end{ex}

\begin{ex}\label{ex:2} Choose $g=\left(\begin{array}{cc}a & b \\ c & d\end{array}\right)$ in $ G$, and put $N:=\det(\mathrm{red}(g))$ (see \textsection \ref{sec:modular_polynomials} for notation). Let $V\subset\mathbb{C}^{4}$ be the affine variety defined as
$$V=\left\{(z_1,z_2,w_1,w_2)\in \mathbb{C}^4:z_{1}(cz_2+d) = az_{2}+b , \Phi_{N}(w_{1},w_{2})+1=0\right\}.$$ 
By the equivalence between (M1) and (M2) in \textsection\ref{sec:modular_polynomials}, $V$ cannot intersect $E_{j}^{2}$. 
\end{ex}

Observe that the cases of Examples \ref{ex:1} and \ref{ex:2}, which fail to have points in $E_{j}^{n}$, are all of algebraic varieties for which the projection map onto the first set of coordinates is not dominant (which is one of the conditions required in Theorem \ref{th:main1}). However, this is not a necessary condition. As we already saw in Example \ref{ex:1}, when $r\in\mathbb{H}$ we have  $V\cap E_{j}^{1}\neq\emptyset$ despite the fact that the projection of $V$ onto the first coordinate is not Zariski dense. 






\section{On certain systems of equations involving automorphic functions}\label{sec:equations_automorphic}

In this section we prove that certain systems of analytic equations involving meromorphic automorphic functions have infinitely many solutions. This is an automorphic analogue of a result of Brownawell and Masser mentioned in the introduction. By a meromorhic automorphic function we mean a meromorphic function $f(z)$ on $\mathbb{H}$ that satisfies the following conditions:
\begin{enumerate}
    \item[(A1)] There exists a Fuchsian group of the first kind $\Gamma_0\subset \mathrm{SL}_2(\mathbb{R})$ such that $f(z)$ is automorphic for $\Gamma_0$, namely, $f(\gamma z)=f(z)$ for every $\gamma$ in $ \Gamma_0$ and every $z$ in $ \mathbb{H}$.
    \item[(A2)] $f(z)$ is meromorphic at each cusp of $\Gamma_0$.
\end{enumerate}
For a precise definition of (A2), we refer the reader to \cite[\S 2.1]{miyake}.

A non-empty subset $U$ of $\mathbb{C}^n$ is called a \emph{domain} if it is open and connected (with respect to the Euclidean topology).

\begin{prop}
\label{prop:masser}
Let $f_{1},\ldots,f_{n}$ be non-constant
meromorphic automorphic functions, let $U\subseteq\mathbb{C}^{n}$ be a 
domain such that $U \cap\mathbb{R}^{n}\neq\emptyset$, and  let $p_{1},\ldots,p_{n}:U\rightarrow\mathbb{C}$ be 
holomorphic functions. Assume that the following condition is  satisfied for every $i$ in $\{1,\ldots,n\}$: if $p_i$ is a constant function, then $f_i$ attains the value of $p_i$ in $\mathbb{H}$.
Then, the system of equations
\begin{equation}
\label{eq:masser}
\begin{array}{ccc}
    f_{1}(z_{1}) &=& p_{1}(z_{1},\ldots,z_{n}), \\
 &\vdots& \\
f_{n}(z_{n}) &=& p_{n}(z_{1},\ldots,z_{n}),
\end{array}
\end{equation}
has infinitely many solutions in $U\cap \mathbb{H}^{n}$.
\end{prop}

We remark that in the exponential case, the proof of \cite[Theorem 2.1]{daquino2} relies on a theorem of Kantorovich which refines Newton's approximation method for finding zeros of vector functions. Our proof, instead, goes on a different direction; we use Rouch\'e's theorem in several variables and standard properties of Fuchsian groups and automorphic functions.

For the convenience of the reader we recall Rouch\'e's theorem  and refer to \cite[Theorem 2 in Chapter IV \textsection 18.55]{shabat} for details. See also \cite[p.~287]{shabat} for the definition of the \emph{order} of a zero of a holomorphic mapping in several variables.

\begin{thm}[Rouch\'e]
\label{thm:rouche}
Let $U\subset \mathbb{C}^n$ be a bounded 
domain with Jordan smooth boundary $\partial U$, and let $f,g:\overline{U}\to \mathbb{C}^n$ be two continuous functions with components $f_i$ and $g_i$, respectively, whose restrictions to $U$ are holomorphic. If at each point $\mathbf{x}$ in $\partial U$ for at least one $i$ in $\{1,\ldots,n\}$ we have 
\begin{equation}
\label{eq:rouche}
    |f_i(\mathbf{x})|>|g_i(\mathbf{x})|,
\end{equation}
then the map $f+g$ has as many zeros (counting orders) as $f$ in $U$.
\end{thm}

Note that, if we have that $\|f(\mathbf{x})\|> \|g(\mathbf{x})\|$ for some $\mathbf{x}$ in $\partial U$, where $\|\cdot\|$ denotes the Euclidean norm\footnote{The Euclidean norm is given by $\|(x_1,\ldots,x_n)\|:=\sqrt{|x_1|^2+\ldots+|x_n|^2}$ for $(x_1,\ldots,x_n)$ in $\mathbb{C}^n$.} on $\mathbb{C}^{n}$, then condition (\ref{eq:rouche}) is satisfied for at least one $i$ in $\{1,\ldots,n\}$ automatically. 

\begin{proof}[Proof of Proposition \ref{prop:masser}]
For simplicity, we fix an integer $m$ in $ \{0,\ldots,n\}$ and assume that $p_1,\ldots,p_m$ are non-constant functions and that $p_{m+1},\ldots,p_n$ are constant, with the obvious conventions if $m=0$ or $m=n$. For each $i$ in $ \{1,\ldots,n\}$ we let $\Gamma_{i}$ denote the Fuchsian group of the first kind with respect to which $f_{i}$ is invariant, and let $X_i$ denote the compactification of the Riemann surface $Y_i:=\Gamma_i\backslash \mathbb{H}$. Consider the Riemann sphere $\widehat{\mathbb{C}}=\mathbb{C}\cup \{\infty\}$ as a compact Riemann surface. Then, for every $i$ in $ \{1,\ldots,m\}$ the function $f_i$ induces a non-constant holomorphic map $F_i:X_i\to \widehat{\mathbb{C}}$. Since $F_i$ is non-constant, we have $F_i(X_i)=\widehat{\mathbb{C}}$. Since the set $X_i\setminus Y_i$ is finite, we conclude that $ F_i(X_i \setminus Y_i)$ is finite. Put $A_i:=\mathbb{C}\setminus F_i(Y_i)$. Then $A_i\subseteq F_i(X_i \setminus Y_i)$, hence $A_i$ is finite. Let $f_i^{-1}(\infty)$ denote the set of poles of $f_i$ in $\mathbb{C}$. Then, we have $ f_i(\mathbb{H}\setminus f_i^{-1}(\infty))=F_i( Y_i)\cap \mathbb{C}=\mathbb{C}\setminus A_i$ for every $i$ in $ \{1,\ldots,m\}$. 
This implies the equality of sets
$$\mathbb{C}^m\setminus \left( f_1(\mathbb{H}\setminus f_1^{-1}(\infty))\times \cdots \times f_m(\mathbb{H}\setminus f_m^{-1}(\infty))\right)=\bigcup_{i=1}^m \{(x_k)_{k=1}^m\in \mathbb{C}^m:x_i\in A_i\}.$$
Now, if the set
$$\{(p_1(\mathbf{x}),\ldots,p_m(\mathbf{x})):\mathbf{x}\in U\cap\mathbb{R}^n\}$$
        were contained in  $\mathbb{C}^m\setminus \left( f_1(\mathbb{H}\setminus f_1^{-1}(\infty))\times \cdots \times f_m(\mathbb{H}\setminus f_m^{-1}(\infty))\right)$, then the  function
$$P:U\to \mathbb{C} \text{ defined as }P(\mathbf{x}) = \prod_{i=1}^m\prod_{a \in A_i}(p_i(\mathbf{x})-a)$$ 
would be a holomorphic function on $U$ vanishing over $U \cap\mathbb{R}^n$. 
This would imply that $P=0$, see e.g.~\cite[p.~21]{shabat}, hence at least one $p_i$ among the functions $p_1,\ldots,p_m$ would be constant, which is a contradiction. This proves that there exists a point $\mathbf{x}_0$ in $ U\cap\mathbb{R}^{n}$ such that for every $i$ in $ \{1,\ldots,m\}$, the automorphic function $f_{i}$ attains the value $p_{i}(\mathbf{x}_0)$. By our hypothesis we conclude that  for every $i$ in $\{1,\ldots,n\}$, $f_{i}$ attains the value $p_{i}(\mathbf{x}_0)$.
Consider the action of $\Gamma_1\times \cdots \times \Gamma_n$ on $\mathbb{H}^n$  given by
$$(g_1,\ldots,g_n)(z_1,\ldots,z_n):=(g_1z_1,\ldots,g_nz_n).$$
Put $\alpha_i:= p_{i}(\mathbf{x}_0), \mathbf{\alpha}:=(\alpha_i)_{i=1}^n$ and choose $\mathbf{w}= \left(w_{i}\right)_{i=1}^{n}$ in $ \mathbb{H}^{n}$  such that $f_{i}(w_{i}) = \alpha_{i}$ for every $i$ in $\{1,\ldots,n\}$. Since $\Gamma_i$ is a Fuchsian group of the first kind, we have that its limit set $\Lambda(\Gamma_i)$  equals 
$\mathbb{R}\cup \{\infty\}$ (see \cite[\S 8.1]{beardon}). By \cite[Theorem 5.3.9]{beardon}, for every point $z$ in $\mathbb{H}$ the set of accumulation points of the $\Gamma_i$-orbit of $z$ equals $\Lambda(\Gamma_i)$. In particular, we can find a sequence $\left(\mathbf{\gamma}_{k}\right)_{k=1}^{\infty}$ of elements $\gamma_k=(\gamma_{k,i})_{i=1}^n$ in $\Gamma_{1}\times\cdots\times\Gamma_{n}$ such that $\|\gamma_{k}\mathbf{w} - \mathbf{x}_0\|$ tends to zero as $k$ tends to infinity (we recall that $\|\cdot\|$ denotes the  Euclidean norm on $\mathbb{C}^{n}$). \\
Let $g:U\rightarrow\mathbb{C}^{n}$ and $f:\mathbb{H}^{n}\rightarrow\mathbb{C}^{n}$ be given by $g(\mathbf{z}) = \left(p_{1}(\mathbf{z}),\ldots,p_{n}(\mathbf{z})\right)$ and $f(\mathbf{z}) = \left(f_{1}(z_1),\ldots,f_{n}(z_n)\right)$ for $\mathbf{z}=\left(z_{i}\right)_{i=1}^{n}$, respectively.
If $m=0$, then for every positive integer $k$ with $\gamma_k \mathbf{w}\in U$ we have $f(\gamma_k\mathbf{w})=g(\gamma_k\mathbf{w})$ and the desired result holds. In what follows, we assume $m\geq 1$. 
Since each $f_i$ is not a constant function, the usual identity theorem from complex analysis implies that for every $i$ in $\{1,\ldots ,m\}$ there exists a small Euclidean closed disk $B_i\subset \mathbb{H}$ around $w_i$ such that $f_i(z)\neq \alpha_i$ for every $z$ in $ B_i\setminus \{w_i\}$. 
Put $B=B_1\times \cdots \times B_n$ and $\delta = \min\left\{\|f(\mathbf{z}) - \alpha\| : \mathbf{z}\in\partial B\right\}$. Denote by $d_{\mathrm{hyp}}$ the hyperbolic distance\footnote{This is the metric given by $d_{\mathrm{hyp}}(z,w)=\mathrm{arcosh}\left(1+\frac{|z-w|^2}{2\mathrm{Im}(z)\mathrm{Im}(w)}\right)$ for $z,w$ in $\mathbb{H}$.} in $\mathbb{H}$ and define
$$d(B):=\max\{d_{\mathrm{hyp}}(z,w):z,w\in B_i,i\in \{1,\ldots,n\}\}.$$
For each $i$ in $\{1,\ldots,n\}$ and each positive integer $k$ the set $\gamma_{k,i} B_i$ is an Euclidean disk, hence there exist unique points $l_{i,k},h_{i,k}$ in $\gamma_{k,i} B_i$ such that
$$\mathrm{Im}(l_{i,k})=\min\{\mathrm{Im}(z):z\in \gamma_{k,i} B_i\}\mbox{ and }\mathrm{Im}(h_{i,k})=\max\{\mathrm{Im}(z):z\in \gamma_{k,i} B_i\}.$$ 
By properties of the hyperbolic distance we have 
$$\log \left(\frac{\mathrm{Im}(h_{i,k})}{\mathrm{Im}(l_{i,k})}\right)= d_{\mathrm{hyp}}(h_{i,k},l_{i,k})=d_{\mathrm{hyp}}(\gamma_{k,i}^{-1}h_{i,k},\gamma_{k,i}^{-1}l_{i,k})\leq d(B),$$
for every $i$ and every $k$. This implies that there exists a positive constant $C$, depending only on $B$, such that $\mathrm{Im}(h_{i,k})\leq C \cdot \mathrm{Im}(l_{i,k})$ for every $i$ and $k$. Since $\mathrm{Im}(l_{i,k}) \leq  \mathrm{Im}(\gamma_{k,i} w_i)$ and $\mathrm{Im}(\gamma_{k,i} w_i)$ tends to zero as $k$ tends to infinity, we conclude that $\mathrm{Im}( h_{i,k})$ also tends to zero as $k$ tends to infinity. This implies that for every Euclidean neighbourhood $W$ of $\mathbf{x}_0$ in $\mathbb{C}^n$ there exists a positive integer $N$ such that $\gamma_kB\subset W$ for every $k>N$. By continuity of $g$, we deduce that there exists a positive integer $M$ such that for every $k>M$ we have $\gamma_k B\subseteq U$ and
$$
  \max \{\|g( \mathbf{x})-\alpha\|:\mathbf{x}\in \gamma_k B\}<\delta.  
$$
Since $\delta = \min\left\{\|f(\mathbf{x})-\alpha\| : \mathbf{x}\in\partial(\gamma_{k}B)\right\}$
and $\partial(\gamma_{k}B)$ is a Jordan boundary, we can apply Rouch\'e's theorem  to the functions $f - \alpha$ and $\alpha - g$ on $\gamma_{k}B$ and conclude that these functions  have the same number of zeros (counting orders) in $\gamma_{k}B$ for every $k>M$. As $f-\alpha$ has a zero there (namely $\gamma_{k}\mathbf{w}$), we conclude that $f(\mathbf{z})=g(\mathbf{z})$ has a solution in $\gamma_k B$. Since for every $i$ we know that $h_{i,k}$ tends to zero as $k$ tends to infinity, we can pass to a subsequence, if necessary, and  assume that the sets $\gamma_k B$ for $k>M$ are all pairwise disjoint. This proves that $f(\mathbf{z})=g(\mathbf{z})$ has infinitely many solutions and completes  the proof of the proposition. 
\end{proof}

\begin{remark}
If we choose each $f_i$ to be the $j$ function in Proposition \ref{prop:masser}, then for every $i$ we have that $f_i$ assumes all values of $p_i$ automatically since $j(\mathbb{H})=\mathbb{C}$. 
\end{remark}


\section{Varieties of triangular form}\label{sec:triangular_forms}

For the proof of Theorem \ref{th:main1}, it will be convenient to consider the following type of varieties. 

\begin{defn}
Let $n,d$ be positive integers with $n\leq d<2n$. An affine irreducible algebraic variety $V\subseteq\mathbb{C}^{2n}$ of dimension $d$ will be called \emph{of triangular form}
if it can be defined by polynomials  $p_{1},\ldots,p_{2n-d}$  satisfying the following two conditions:
\begin{enumerate}
        \item[(i)] For every $i$ in $\left\{1,\ldots,2n-d\right\}$ there exist an integer $d_i\geq 1$, polynomials $q_{i,0},\ldots,q_{i,d_i-1}$ in $\mathbb{C}[X_1,\ldots,X_n,Y_1,\ldots,Y_{i-1}]$ and $q_{i,d_i}$ in $\mathbb{C}[X_1,\ldots,X_n]$ non-zero such that
\begin{equation}\label{eq:eqtriangularform}
    p_i=q_{i,d_i}Y_i^{d_i}+\sum_{k=0}^{d_i-1}q_{i,k}Y_i^k.
\end{equation}
In particular, each $p_{i}$ depends only on $Y_1,\ldots,Y_{i}$ among the variables $Y_{1},\ldots,Y_{n}$.
\item[(ii)] If we define 
$\pi_{V}:V\to \mathbb{C}^d$ by
$$\pi_V(x_1,\ldots,x_n,y_1,\ldots,y_n)=(x_1,\ldots,x_n,y_{2n-d+1},\ldots,y_n),$$
then $\mathrm{deg}(\pi_{V})=d_1\cdots d_{2n-d}$.
\end{enumerate}
\end{defn}
Given an algebraic variety $V\subseteq \mathbb{C}^{2n}$ of triangular form 
as above, we define
$$V_{0}=\{(x_1,\ldots,x_n,y_1,\ldots,y_n)\in V: q_{i,d_i}(x_1,\ldots,x_n)\neq 0 \mbox{ for every } i\text{ in }\{1,\ldots,2n-d\}\}.$$
Note that $V_0$ is a non-empty Zariski open subset of $V$ and $\pi_{V}(V_0)$ is a non-empty Zariski open subset of $\mathbb{C}^d$. In particular, $\pi_{V}$ is dominant and  $\mathrm{deg}(\pi_{V})$ is well defined. 

\begin{remark}
For a general irreducible algebraic variety $V\subseteq \mathbb{C}^{2n}$ defined by polynomials $p_1,\ldots,p_{2n-d}$ satisfying \eqref{eq:eqtriangularform} it might happen that $\mathrm{deg}(\pi_{V})< d_1\cdots d_{2n-d}$.  An example is given by the variety
$$V=V(Y_1-X_1^2,Y_2^2+2Y_2X_1+Y_1)\subset \mathbb{C}^4.$$
Note that $V=V(Y_1-X_1^2,Y_2+X_1)$ and $\mathrm{deg}(\pi_{V})=1$.
\end{remark}

\begin{lem}\label{lem:curves_are_triangular}
Let $V\subset \mathbb{C}^2$ be an irreducible algebraic curve that is not a vertical line. Then, $V$ is of triangular form.
\end{lem}
\begin{proof}
Write $V=V(p)$ with $p(X,Y)$ in $\mathbb{C}[X,Y]$ irreducible. Let us write
$$p(X,Y)=\sum_{i=0}^{d_Y}p_i(X)Y^i,$$
where $p_0(X),\ldots,p_{d_Y}(X)\in \mathbb{C}[X]$ and $p_{d_Y}(X)\neq 0$, so $d_Y$ is the degree of $p$ in the $Y$ variable. Since $V$ is not a vertical line, we have $\frac{\partial p}{\partial Y}\neq 0$, hence  $p$ satisfies \eqref{eq:eqtriangularform} with $d_1=d_Y$. Put $R:=\mathbb{C}[X]$ and define $r$ as the resultant of $p$ and $\frac{\partial p}{\partial Y}$ as polynomials in $R[Y]$, see e.g.~\cite[Chapter 2,  \S2]{griffiths}. Since $p$ is irreducible and $\frac{\partial p}{\partial Y}\neq 0$, we have by \cite[Theorem 2.2]{griffiths} that $r$ is a non-zero element of $R$. It follows that the set
$$A:=\{z\in \mathbb{C}:p_{d_Y}(z)r(z)\neq 0\}$$
is a non-empty Zariski open subset of $\mathbb{C}$. Note that $A$ is contained in the image of the map $\pi_V:V\to \mathbb{C}$, $\pi_V(x,y)= x$. By \cite[Corollary 2.4]{griffiths} there exist polynomials $f,h$ in $R[Y]$ with $fp+h\frac{\partial p}{\partial Y}=r$. This implies that for every $z$ in $ A$ and every $(w_1,w_2)$ in $ \pi_V^{-1}(z)$ we have  $\frac{\partial p}{\partial Y}(w_1,w_2)\neq 0$, hence $z$ has exactly $d_Y$ preimages under $\pi_V$. This proves that $\mathrm{deg}(\pi_V)=d_Y$ and completes the proof of the lemma.
\end{proof}

The main properties of algebraic varieties of triangular form that we are going to use are summarized in the following proposition. For convenience,  for every set $A$ we define $A\times \mathbb{C}^0:=A$.

\begin{prop}\label{prop:conj1fortriangularform}
Let $n,d$ be positive integers with $n\leq d<2n$ and let $V\subseteq \mathbb{C}^{2n}$ be an  algebraic variety of dimension $d$ of triangular form. Then, there exists a non-empty Zariski open subset $B$ of $\mathbb{C}^d$ such that the following properties hold.
\begin{enumerate}
    \item  $B\subseteq \pi_{V}(V_0)$ and $B\cap \mathbb{R}^d$ contains a non-empty Zariski open subset of $ \mathbb{R}^d$.
    \item For every point $(z_i)_{i=1}^{2n}$ in $ \pi_{V}^{-1}(B)$, there exist Euclidean neighbourhoods $U_1,U_2$ of $(z_1,\ldots,z_n)$ in $\mathbb{C}^n$ and of $(z_{n+1},\ldots,z_{3n-d})$ in $ \mathbb{C}^{2n-d}$, respectively, and a holomorphic function $H:U_1\to U_2$ such that
$$ V\cap (U_1\times U_2\times \mathbb{C}^{d-n})=\{(\mathbf{w}_1,H(\mathbf{w}_1),\mathbf{w}_2):\mathbf{w}_1\in U_1,\mathbf{w}_2\in \mathbb{C}^{d-n}\}$$
if $d>n$, and
$$ V\cap (U_1\times U_2)=\{(\mathbf{w}_1,H(\mathbf{w}_1)):\mathbf{w}_1\in U_1\}$$
if $d=n$.
\end{enumerate}
Moreover, if $J$ is a finite subset of $\mathbb{C}^{2n-d}$ such that $V$ is not contained in $\mathbb{C}^n\times J \times \mathbb{C}^{d-n}$, then we have the following properties.
\begin{enumerate}
    \item[(3)]  For every point $(z_i)_{i=1}^{2n}$ in $ \pi_{V}^{-1}(B\cap \mathbb{R}^d)$ and every triple $U_1,U_2,H$ as in part (2), 
    there exists an Euclidean open subset $U_1'\subseteq U_1$ 
    such that $U_1' \cap \mathbb{R}^d\neq \emptyset$ and $H(U_1')\cap J=\emptyset$. 
    \item[(4)] For every point $\mathbf{a}$ in $  B\cap \mathbb{R}^d$ and every Euclidean neighbourhood $U$ of $\mathbf{a}$ in $\mathbb{C}^d$, the set $$E_j^n\cap \pi_{V}^{-1}(U) \cap (\mathbb{C}^{2n}\setminus (\mathbb{C}^n\times J \times \mathbb{C}^{d-n}))$$
is infinite.  \end{enumerate}
In particular, there exist infinitely many points in the set $E_{j}^{n}\cap V_0\cap (\mathbb{C}^{2n}\setminus (\mathbb{C}^n\times J \times \mathbb{C}^{d-n}))$.
\end{prop}

\begin{remark}
If we put $J=\emptyset$ then $\mathbb{C}^n\times J \times \mathbb{C}^{d-n}=\emptyset$ and by Proposition \ref{prop:conj1fortriangularform} we have that for every algebraic variety $V\subseteq \mathbb{C}^{2n}$ of triangular form the set $E_j^n\cap V_0$ is infinite.
\end{remark}

Our proof of Proposition \ref{prop:conj1fortriangularform} makes use of the Implicit Function Theorem. Since this result is also used in the proofs of Theorems \ref{th:main3} and \ref{thm:main4}, we recall its formulation for the convenience of the reader (see \cite[\textsection 4.9, Theorem 3]{shabat} for details).

\begin{thm}[Implicit Function Theorem]
Let $B$ be a non-empty open subset of $\mathbb{C}^{n}\times\mathbb{C}^{m}$, $F=(F_{1},\ldots,F_{m}):B\rightarrow\mathbb{C}^{m}$ be a holomorphic map on the variables $(z_{1},\ldots,z_{n+m})$, and $(\mathbf{x}_{0},\mathbf{y}_{0})$ in $\mathbb{C}^{n}\times\mathbb{C}^{m}$ be a point in $B$ satisfying $F(\mathbf{x}_{0},\mathbf{y}_{0}) = 0$ and
\begin{equation}
\label{eq:jacobian}
    \det\left(\frac{\partial F_{\mu}}{\partial z_{\nu}} (\mathbf{x}_{0},\mathbf{y}_{0})\right)_{\mu = 1,\ldots,m;\, \nu = n+1,\ldots,n+m}\neq 0 .
\end{equation}
Then there is an open neighbourhood $U = U_{1}\times U_{2}$ of $(\mathbf{x}_{0},\mathbf{y}_{0})$ contained in $B$ and a holomorphic map $H:U_{1}\rightarrow U_{2}$ such that
\begin{equation*}
    \left\{(\mathbf{x},\mathbf{y})\in U : F(\mathbf{x},\mathbf{y}) = 0\right\} = \left\{(\mathbf{x},H(\mathbf{x})) : \mathbf{x}\in U_{1}\right\}.
\end{equation*}
\end{thm}

We now give the proof of Proposition \ref{prop:conj1fortriangularform}.

\begin{proof}[Proof of Proposition \ref{prop:conj1fortriangularform}]
Let $p_{1},\ldots,p_{2n-d}$ be polynomials in $\mathbb{C}[X_{1},\ldots,X_{n},Y_{1},\ldots,Y_{2n-d}]$ defining $V$ satisfying the conditions of triangular form, let $q_{i,0},\ldots,q_{i,d_i}$ be the polynomials satisfying \eqref{eq:eqtriangularform} and put $D:=\deg(\pi_V)=d_1\cdots d_{2n-d}$. In the case $d>n$ define
\begin{equation*}
B_0:=\{(x_1,\ldots,x_n,y_{2n-d+1},\ldots,y_{n})\in \mathbb{C}^{d}:q_{i,d_i}(x_1,\ldots,x_n) \neq 0 \text{ for every }i \text{ in } \{1,\ldots,2n-d\}\},    
\end{equation*}
and in the case $d=n$ define
\begin{equation*}
B_0:=\{(x_1,\ldots,x_n)\in \mathbb{C}^{n}:q_{i,d_i}(x_1,\ldots,x_n) \neq 0 \text{ for every }i \text{ in } \{1,\ldots,n\}\}.    
\end{equation*}
Then, $B_0$ is a non-empty  Zariski open subset of $\mathbb{C}^{d}$. Since $\mathbb{C}$ is algebraically closed, we have $B_0\subseteq \pi_{V}(V)$. By \cite[Theorem 2.29]{Shafarevich1} we can find a non-empty Zariski open subset $B\subseteq B_0$ such that for every $\mathbf{z} $ in $ B$ , $\pi_{V}$ is unramified over $\mathbf{z}$, meaning that $\mathbf{z}$ has exactly $D$ preimages under $\pi_{V}$. In order to prove $(1)$, we have to check that $B\cap \mathbb{R}^d$ contains a non-empty Zariski open subset of $ \mathbb{R}^d$. But this is standard; since the complement of $B$ in $\mathbb{C}^{d}$ is a  Zariski closed proper subset, it must be equal to the set of zeros of a finite number of non-zero polynomials $Q_1,\ldots,Q_s$ in $ \mathbb{C}[X_1,\ldots,X_n,Y_{2n-d+1},\ldots,Y_{n}]$. We have $\mathbb{C}^d\setminus V(Q_1)\subseteq B$. Since $Q_1$ is non-zero and $ \mathbb{R}$ is infinite, there  exists $\mathbf{z}_0$ in $\mathbb{R}^d$ with $Q_1(\mathbf{z}_0)\neq 0$, hence $  \mathbb{R}^d\setminus V(Q_1)$ is a non-empty Zariski open subset of $ \mathbb{R}^d$ contained in $B$. This proves $(1)$.\\
%
%
%
%
In order to prove (2), let $\mathbf{w}_0$ be a point in $B$ (hence $\mathbf{w}_0\in B_0$) with first $n$ coordinates equal to $x_1,\ldots,x_n$.  Note that, starting with the conditions $X_i=x_i$ for $i$ in $ \{1,\ldots,n\}$ we can solve \eqref{eq:eqtriangularform} for each $i$ in $ \{1,\ldots,2n-d\}$ as a system of equations on the variables $Y_1,\ldots,Y_{2n-d}$, obtaining at most $D$ different solutions $(y_1,\ldots,y_{2n-d})$ in $\mathbb{C}^{2n-d}$. Moreover, we get exactly $D$ different solutions if and only if each partial derivative $\frac{\partial p_i}{\partial Y_i}$, for $i$ in $\{1,\ldots,2n-d\}$, does not vanish at any point of the form $(x_1,\ldots,x_n,y_1,\ldots,y_{2n-d})$ with $(y_1,\ldots,y_{2n-d})$ a solution of the system. Since $\mathbf{w}_0$ has exactly $D$ preimages under $\pi_{V}$, we conclude
$$\det\left(\frac{\partial p_i}{\partial Y_j}(\mathbf{z}_0)\right)_{i,j=1,\ldots,2n-d}=\prod_{i=1}^{2n-d}\frac{\partial p_i}{\partial Y_i}(\mathbf{z}_0) \neq 0,$$
for every $\mathbf{z}_0$ in $ \pi_{V}^{-1}(\mathbf{w}_0)$. Hence, for every point $(z_1,\ldots,z_{2n})$ in $ \pi_{V}^{-1}(\mathbf{w}_0)$ we can apply the Implicit Function Theorem to the map $F:\mathbb{C}^{3n-d}\to \mathbb{C}^{2n-d}$ given by
$$F(w_1,\ldots,w_{3n-d}):=(p_1(w_1,\ldots,w_{n+1}),\ldots,p_{2n-d}(w_1,\ldots,w_{3n-d})),$$
at the point $(z_1,\ldots,z_{3n-d})$. This way, we get the existence of Euclidean neighbourhoods $U_1,U_2$ of $(z_1,\ldots,z_n)$ in $\mathbb{C}^n$ and of $(z_{n+1},\ldots,z_{3n-d})$ in $ \mathbb{C}^{2n-d}$, respectively, and a holomorphic function $H=(H_1,\ldots,H_{2n-d}):U_1\to U_2$ such that
$$ V\cap (U_1\times U_2\times \mathbb{C}^{d-n})=\{(\mathbf{w}_1,H(\mathbf{w}_1),\mathbf{w}_2)):\mathbf{w}_1\in U_1,\mathbf{w}_2\in \mathbb{C}^{d-n}\}$$
if $d>n$, and
$$ V\cap (U_1\times U_2)=\{(\mathbf{w}_1,H(\mathbf{w}_1)):\mathbf{w}_1\in U_1\}$$
if $d=n$. This proves (2). Now, let $J$ be a finite subset of $\mathbb{C}^{2n-d}$ such that $V$ is not contained in $\mathbb{C}^n\times J \times \mathbb{C}^{d-n}$ and choose a point  $\mathbf{u}_0$ in $ \pi_V^{-1}(B\cap \mathbb{R}^d)$. Let $U_1,U_2$ and $H=(H_1,\ldots,H_{2n-d})$ be as in part (2) with $(z_i)_{i=1}^{2n}=\mathbf{u}_0$. By shrinking $U_1$ if necessary, we can assume that it is connected. We claim that
$$\{H(\mathbf{z}):\mathbf{z}\in U_1\cap \mathbb{R}^n\}\not \subseteq J.$$
Indeed, assume that this is not the case. Since $J$ is finite, $H(z)$ is holomorphic and $U_1$ is connected, we get $H(U)=\{\mathbf{s}\}$ for some $\mathbf{s}$ in $J$ (see e.g.~\cite[p.~21]{shabat}), hence
$$ V\cap (U_1\times U_2\times \mathbb{C}^{d-n})=U_1\times \{\mathbf{s}\}\times \mathbb{C}^{d-n}\subseteq V\cap (\mathbb{C}^n \times \{\mathbf{s}\}\times \mathbb{C}^{d-n}).$$ 
This implies, by Corollary \ref{cor:generic_contention}, that $V$ is contained in $\mathbb{C}^n \times \{\mathbf{s}\}\times \mathbb{C}^{d-n}$, which contradicts our hypothesis. This proves our claim. It follows that we can find $\mathbf{v}_0$ in $ U_1\cap \mathbb{R}^n$ and $i_0$ in $ \{1,\ldots,2n-d\}$ such that $H_{i_0}(\mathbf{v}_0)\not \in J$. Take an Euclidean open set $U_1'\subseteq U_1$ such that $\mathbf{v}_0\in U'_1$ and $H_{i_0}(\mathbf{w})\not \in J$ for every $\mathbf{w}$ in $U'_1$. It is clear that $U_1'$ satisfies the desired properties. This proves part (3).\\
In order to prove (4), let $\mathbf{a}$ be a point in $B\cap \mathbb{R}^d$, let $U$ be an Euclidean neighbourhood of $\mathbf{a}$ in $\mathbb{C}^d$ and choose a point $\mathbf{t}_0$ in $\pi_V^{-1}(\mathbf{a})$.  Let $U_1,U_2$ and $H=(H_1,\ldots,H_{2n-d})$ be given by part (2) with $(z_i)_{i=1}^{2n}=\mathbf{t}_0$. By shrinking $U_1$, if necessary, we can assume that in the case $d>n$ there exists an open subset $U_3$ of $\mathbb{C}^{d-n}$ such that $$V\cap (U_1\times U_2\times U_3)\subseteq \pi_V^{-1}(U),$$
and that in the case $d=n$ we have
$$V\cap (U_1\times U_2)\subseteq \pi_V^{-1}(U).$$
Let $U_1'$ be the open subset of $ U_1$ given by  part (3). If $d>n$, fix a point $(\alpha_1,\ldots,\alpha_{d-n})$ in $U_3$. 
Since $ U_1'\cap \mathbb{R}^n\neq \emptyset$ and $j(\mathbb{H})=\mathbb{C}$, we can apply Proposition \ref{prop:masser} to the system of equations
\begin{equation}\label{systemeq}
\begin{array}{ccc}
    j(w_{1}) &=& H_1(w_{1},\ldots,w_n), \\
 &\vdots& \\
j(w_{2n-d}) &=& H_{2n-d}(w_{1},\ldots,w_n),\\
j(w_{2n-d+1}) &=& \alpha_1,\\
& \vdots & \\
j(w_{n}) &=& \alpha_{d-n},
\end{array}
\end{equation}
for $(w_1,\ldots,w_n)$ in $ U_1'\cap \mathbb{H}^n$, where only the first $n$ equations should be considered in the case $d=n$. For each solution $\mathbf{w}=(w_1,\ldots,w_n)$ of this system of equations we have that
$$(\mathbf{w},j(\mathbf{w}))=(w_1,\ldots,w_n,H_1(\mathbf{w}),\ldots,H_{2n-d}(\mathbf{w}),\alpha_1,\ldots,\alpha_{d-n})$$
if $d>n$, and
$$(\mathbf{w},j(\mathbf{w}))=(w_1,\ldots,w_n,H_1(\mathbf{w}),\ldots,H_{2n-d}(\mathbf{w}))$$
if $d=n$. It follows that $(\mathbf{w},j(\mathbf{w}))$ is in $E_j^n\cap \pi_{V}^{-1}(U)\cap (\mathbb{C}^{2n}\setminus (\mathbb{C}^n\times J \times \mathbb{C}^{d-n}))$ for every solution $\mathbf{w}$. This proves (4) since the system \eqref{systemeq} has infinitely many solutions in $U'_1 \cap \mathbb{H}^n$. \\
Finally, taking $U=B$ in part (4), and noting that $\pi_V^{-1}(B)\subseteq \pi_V^{-1}(\pi_V(V_0))=V_0$, we conclude that  $E_j^n\cap V_0\cap (\mathbb{C}^{2n}\setminus (\mathbb{C}^n\times J \times \mathbb{C}^{d-n}))$ is infinite. This completes the proof of the proposition.
\end{proof}




\begin{cor}
\label{cor:dense1}
Let $V\subseteq\mathbb{C}^{2n}$ be an  algebraic variety of triangular form of dimension $d\geq n$. 
Then $\pi_{V}\left(E_{j}^{n}\cap V_0\right)$ is Zariski dense in $\mathbb{C}^{d}$.
\end{cor}
\begin{proof}
Let $B$ be the non-empty Zariski open subset of $\mathbb{C}^d$ contained in $\pi_V(V_0)$ given by Proposition \ref{prop:conj1fortriangularform} and let $A$ be a non-empty Zariski open subset of $B\cap \mathbb{R}^d$.  Let $\mathbf{a}$ be a point in $A$ and let $\mathcal{U}$ be a Zariski open subset of $\mathbb{C}^d$ contained in $B$ with $\mathbf{a} \in \mathcal{U}$. Since $\mathcal{U}$ is also an open subset of $\mathbb{C}^d$ in the Euclidean topology, we can use  Proposition \ref{prop:conj1fortriangularform}(4) with $J=\emptyset$ and find a point $\mathbf{z}$ in $\pi_{V}^{-1}(\mathcal{U})\cap E_j^n$. Since $\pi_{V}^{-1}(\mathcal{U})\subseteq \pi_{V}^{-1}(B)\subseteq V_0$, we have $\mathbf{z}\in V_0$. This implies that $\pi_{V}(E_j^n\cap V_0)\cap \mathcal{U}$ is non-empty. Thus, $A$ is contained in the Zariski closure of $\pi_{V}(E_j^n\cap V_0)$. Since $A$ is Zariski dense in $\mathbb{R}^d$ and $\mathbb{R}^{d}$ is Zariski dense in $\mathbb{C}^{d}$, we have that $A$ is Zariski dense in $\mathbb{C}^d$. This implies that $\pi_{V}(E_j^n\cap V_0)$ is Zariski dense in $\mathbb{C}^d$ and completes the proof of the corollary. 
\end{proof}

\section{Proof of Theorem \ref{th:main1}}\label{sec:proof_of_thm_1}

The proof of Theorem \ref{th:main1} is based on Corollary \ref{cor:dense1} and the following result.

\begin{prop}
\label{prop:dense2}
Let $V\subseteq \mathbb{C}^{2n}$ be an algebraic variety and let $\pi:\mathbb{C}^{2n}\rightarrow\mathbb{C}^{n}$ be the projection onto the first $n$ coordinates. If $\pi(V)$ is Zariski dense in $\mathbb{C}^{n}$, then there exists an algebraic variety $W\subset \mathbb{C}^{2n}$ of dimension $n$ of triangular form  with  $W\subseteq V$. 
\end{prop}
\begin{proof}
Let $p_1,\ldots,p_m$ be polynomials in $\mathbb{C}[X_1,\ldots,X_n,Y_1,\ldots,Y_n]$ defining  $V$  and let $F$ be the field generated by the coefficients of $p_1,\ldots,p_m$. Since the set $\pi(V)$ is dense in $\mathbb{C}^{n}$, it  contains a Zariski open subset of $\mathbb{C}^n$ (see, e.g~\cite[Theorem 1.14]{Shafarevich1}).
Since  $\text{tr.deg.}_{\mathbb{Q}}(\mathbb{C})$ is infinite, we can find a point $\mathbf{x}_0=(x_1,\ldots,x_n)$ in  $\pi (V)$ such that $\mathrm{tr.deg.}_{F}(F(\mathbf{x}_0))=n$. Let $L$ denote the algebraic closure of $F(\mathbf{x}_0)$ in $\mathbb{C}$ and for each $i$ in $\{1,\ldots,m\}$ consider the polynomials $q_i(Y_1,\ldots,Y_n):=p_i(\mathbf{x}_0,Y_1,\ldots,Y_n)$ in $ L[Y_1,\ldots,Y_n]$. If the set
$$V_L(q_1,\ldots,q_m):=\{\mathbf{y}\in L^n:q_i(\mathbf{y})=0 \text{ for every }i\text{ in } \{1,\ldots,m\}\}$$
were empty, then the ideal of $L[Y_1,\ldots,Y_n]$ generated by $q_1,\ldots,q_m$ would contain $1$. But this would imply that the algebraic subset of $\mathbb{C}^n$ defined by the polynomials $q_1,\ldots,q_m$ is empty, contradicting the fact that $\mathbf{x}_0\in \pi(V)$. Hence, we can choose a point $\mathbf{y}_0=(y_1,\ldots,y_n)$ in $V_L(q_1,\ldots,q_n)$, meaning that $\mathbf{y}_0\in L^n$ and $(\mathbf{x}_0,\mathbf{y}_0)\in V$.\\
Put $R:= \overline{F}[X_1,\ldots,X_n]$. The minimal polynomial of $y_{1}$ over $\overline{F}(\mathbf{x}_0)$ is of the form $a_1 g_1(\mathbf{x}_0,Y_1)$ where $g_1(X_1,\ldots,X_n,Y_1)$ is a primitive irreducible polynomial in $R[Y_1]$ of positive degree and $a_1$ is a non-zero element of $\overline{F}(\mathbf{x}_0)$. Similarly, for every $i$ in $\left\{1,\ldots,n\right\}$ with $i>1$  the minimal polynomial of $y_{i}$ over $\overline{F}(\mathbf{x}_0,y_1,\ldots,y_{i-1})$ is of the form $a_i g_i(\mathbf{x}_0,y_1,\ldots,y_{i-1},Y_i)$ where $g_i(X_1,\ldots,X_n,Y_1,\ldots,Y_i)$ is a primitive  irreducible polynomial in $R_i[Y_i]$ of positive degree, where $R_i:=\overline{F}[X_1,\ldots,X_n,Y_1,\ldots,Y_{i-1}]$, whose leading coefficient is in $R$, and $a_i$ is a non-zero element of $\overline{F}(\mathbf{x}_0)$. \\
Let $W$ denote the algebraic subset $V(g_{1}, \ldots g_n)$ of $\mathbb{C}^{2n}$. It satisfies condition (i) from the definition of triangular form by construction. Note that $(\mathbf{x}_0,\mathbf{y}_0)\in V\cap W$ and $\mathrm{tr.deg.}_{\overline{F}}(\overline{F}(\mathbf{x}_0,\mathbf{y}_0))=n$. We claim that $W\subseteq V$. By Lemma \ref{lem:generic}, it is enough to check that $W\cap \overline{F}^{2n}$  is irreducible and of dimension $n$ as algebraic subvariety of $\overline{F}^{2n}$. First, note that the ring isomorphism $\overline{F}[X_1,\ldots,X_n,Y_1,\ldots,Y_n]\to \overline{F}[\mathbf{x}_0,Y_1,\ldots,Y_n]$ given by $f(X_1,\ldots,X_n,Y_1,\ldots,Y_n)\mapsto f(\mathbf{x}_0,Y_1,\ldots,Y_n)$ induces an injective ring morphism
\begin{equation}\label{eq:embedding}
  \overline{F}[X_1,\ldots,X_n,Y_1,\ldots,Y_n]/(g_1,\ldots,g_n)
\hookrightarrow 
\overline{F}(\mathbf{x}_0)[Y_1,\ldots,Y_n]/(\alpha_{1},\ldots,\alpha_{n}),  
\end{equation}
where $\alpha_{i}(Y_1,\ldots, Y_i):=g_i(\mathbf{x}_0,Y_1,\ldots,Y_i)$. Since replacing $Y_1$ by $y_1$ gives a ring isomorphism $\overline{F}(\mathbf{x}_0)[Y_1,\ldots,Y_n]/(\alpha_1)\simeq \overline{F}(\mathbf{x}_0,y_1)[Y_2,\ldots,Y_n]$, we have
$$\overline{F}(\mathbf{x}_0)[Y_1,\ldots,Y_n]/(\alpha_{1},\ldots,\alpha_{n})\simeq \frac{\overline{F}(\mathbf{x}_0)[Y_1,\ldots,Y_n]/(\alpha_{1})}{(\alpha_{1},\ldots,\alpha_{n})/(\alpha_1)}\simeq \overline{F}(\mathbf{x}_0,y_1)[Y_2,\ldots,Y_n]/(\beta_2,\ldots,\beta_n),$$
where $\beta_i(Y_2,\ldots,Y_i):=g_i(\mathbf{x}_0,y_1,Y_2,\ldots,Y_n)$ for $i$ in $\{2,\ldots,n\}$ provided $n\geq 2$. Repeating this argument we obtain that
$$\overline{F}(\mathbf{x}_0)[Y_1,\ldots,Y_n]/(\alpha_{1},\ldots,\alpha_{n})\simeq \overline{F}(\mathbf{x}_0)[y_1,\ldots,y_n].$$
By \eqref{eq:embedding} we conclude that $\overline{F}[X_1,\ldots,X_n,Y_1,\ldots,Y_n]/(g_1,\ldots,g_n)$ is an integral domain. Hence, the polynomials $g_1,\ldots,g_n$ generate a prime ideal of $\overline{F}[X_1,\ldots,X_n,Y_1,\ldots,Y_n]$ and therefore $W\cap \overline{F}^{2n}$ is irreducible in $\overline{F}^{2n}$. Note that the image of the ring morphism \eqref{eq:embedding} equals $\overline{F}[\mathbf{x}_0,Y_1,\ldots,Y_n]/(\alpha_{1},\ldots,\alpha_{n})$, hence the field of fractions of $ \overline{F}[X_1,\ldots,X_n,Y_1,\ldots,Y_n]/(g_1,\ldots,g_n)$ is isomorphic to $\overline{F}(\mathbf{x}_0)[y_1,\ldots,y_n]$. Since this field has transcendence degree over $\overline{F}$ equal to $n$, we conclude that $W\cap \overline{F}^{2n}$ has dimension $n$. This completes the proof of the contention $W\subseteq V$. \\
%
Now, consider the morphism $\pi_{W}|_{W\cap \overline{F}^{2n}}:W\cap \overline{F}^{2n}\to \overline{F}^d$ obtained by restricting $\pi_{W}$ to $W\cap \overline{F}^{2n}$. Its degree equals
$$\left[\overline{F}(\mathbf{x}_0)[y_1,\ldots,y_n]:\overline{F}(\mathbf{x}_0)\right]=\prod_{i=i}^n \left[\overline{F}(\mathbf{x}_0)[y_1,\ldots,y_i]:\overline{F}(\mathbf{x}_0)[y_1,\ldots,y_{i-1}]\right],$$
where $\overline{F}(\mathbf{x}_0)[y_1,\ldots,y_{i-1}]$ is defined as $\overline{F}(\mathbf{x}_0)$ when $i=1$. By construction, we have $\left[\overline{F}(\mathbf{x}_0)[y_1,\ldots,y_i]:\overline{F}(\mathbf{x}_0)[y_1,\ldots,y_{i-1}]\right]= \mathrm{deg}_{Y_i}(g_i)$ for every $i$ in $\{1,\ldots,n\}$, hence  $\pi_{W}|_{W\cap \overline{F}^{2n}}$ has degree $D:=\prod_{i=1}^n \mathrm{deg}_{Y_i}(g_i)$. Since $\overline{F}$ is algebraically closed, we have that $W$ is also irreducible in $\mathbb{C}^{2n}$ (see, e.g.~\cite[Exercise II.3.15]{hartshorne}) and $\mathrm{deg}(\pi_W)=D$. This proves that $W$ satisfies condition (ii) from the definition of triangular form and completes the proof of the proposition.

\end{proof}

\begin{proof}[Proof of Theorem \ref{th:main1}]
By Proposition \ref{prop:dense2} there exists   an algebraic variety of triangular form $W\subseteq \mathbb{C}^{2n}$ of dimension $n$ contained in $V$. By Corollary \ref{cor:dense1} we have that $\pi_{W}(E_j^n\cap W_0)$ is Zariski dense in $\mathbb{C}^n$. This implies the desired result since $\pi_{W}(E_j^n\cap W_0)\subseteq \pi(E_j^n\cap V)$.
\end{proof}

\section{Proof of Theorem \ref{th-main-2}}\label{sec:proof_of_thm_2}

The proof of Theorem \ref{th-main-2} can be found at the end of this section after a few intermediate results.

\subsection{Avoiding special points}

Here we prove that, in a given algebraic variety $V$ of triangular form, the number of $\Gamma$-orbits of the coordinates of special points lying in $V_0$ is bounded. This is Proposition \ref{prop:special} below. First, we recall an elementary technical lemma.

Given a finitely generated field extension $L\supseteq K$, we define $[L:K]_{\mathrm{alg}}$ as the smallest positive integer $n$ for which there exists a field $K_0$ satisfying that $L\supseteq K_0\supseteq K$, $[L:K_0]=n$ and $K_0$ is purely transcendental over $K$.

\begin{lem}\label{lem:bound-degree}
Let $C$ be a finite subset of $\mathbb{C}$, let $x$ be an element of $\mathbb{C}$ and let $M$ be a positive integer. Then, there exists a positive integer $M'$ depending only on $M$ and on  $[\mathbb{Q}(C,x): \mathbb{Q}(x)]_{\mathrm{alg}}$ such that for every 
$y$ in $\mathbb{C}$ that is algebraic over $\mathbb{Q}(x)$ and satisfies  $[\mathbb{Q}(C,x,y):\mathbb{Q}(C,x)]\leq M$  we have $[\mathbb{Q}(x,y):\mathbb{Q}(x)]\leq M'$. 
\end{lem}
\begin{proof}
For any subfield $K$ of $\mathbb{C}$ containing $x$ and any complex number $t$   that is transcendental over $K$, we have 
$[K(t,y):K(t)] = [K(y):K]$ (see, e.g.~\cite[Chapter VIII, Lemma 4.10]{lang3}). 
Now, let $t_{1},\ldots,t_{m}$ in $ \mathbb{C}$ form a transcendence basis for $\mathbb{Q}(C,x)$ over $\mathbb{Q}(x)$ such that the degree $D=[\mathbb{Q}(C,x):\mathbb{Q}(t_{1},\ldots,t_{m},x)]$ equals $[\mathbb{Q}(C,x): \mathbb{Q}(x)]_{\mathrm{alg}}$. Then, as we have just seen, we have $$[\mathbb{Q}(t_{1},\ldots,t_{m},x,y):\mathbb{Q}(t_{1},\ldots,t_{m},x)] = [\mathbb{Q}(x,y):\mathbb{Q}(x)].$$ 
Put $M'=MD$ and let $\mathbf{t} = (t_{1},\ldots,t_{m})$. By the multiplicative property of the degree of field extensions, we get
\begin{eqnarray*}
\left[\mathbb{Q}\left(\mathbf{t},x,y\right):\mathbb{Q}\left(\mathbf{t},x\right)\right] &\leq &
 \left[\mathbb{Q}(C,x,y):\mathbb{Q}\left(\mathbf{t},x\right)\right]\\
 &=& \left[\mathbb{Q}(C,x,y):\mathbb{Q}(C,x)\right]\cdot \left[\mathbb{Q}(C,x):\mathbb{Q}\left(\mathbf{t},x\right)\right]\\
 &\leq & M\cdot D\\
 &=& M'.
\end{eqnarray*}
This proves the  lemma.
\end{proof}

\begin{prop}
\label{prop:special}
Let $V\subseteq\mathbb{C}^{2n}$ be a variety of triangular form of dimension $n$. Then there is a finite set $S\subset\mathbb{H}$ of special points, such that for every special point $\mathbf{z}$ in $\mathbb{H}^{n}$ with $\left(\mathbf{z},j(\mathbf{z})\right)\in V_{0}$, we have that the coordinates of $\mathbf{z}$ are in $\Gamma\cdot S$. 
\end{prop}
\begin{proof}
Let $p_{1},\ldots,p_{n}$ be polynomials defining $V$ satisfying the conditions of triangular form and let $C_0$ be the set of coefficients of the $p_{i}$'s.  Suppose that $\mathbf{z} = (z_{1},\ldots,z_{n})$ in $ \mathbb{H}^{n}$ is special and satisfies $\left(\mathbf{z},j(\mathbf{z})\right)\in V_{0}$. Put $C:=C_0\cup \{z_1,\ldots,z_n\}$ and $D:=d_1\cdots d_n$. Since $[\mathbb{Q}(\mathbf{z}):\mathbb{Q}]\leq 2^n$, we have $[\mathbb{Q}(C):\mathbb{Q}(z_{i})]_{\mathrm{alg}}\leq 2^n[\mathbb{Q}(C_0):\mathbb{Q}]_{\mathrm{alg}}$ and $[\mathbb{Q}(C,j(z_i)):\mathbb{Q}(C)]\leq D$ for every $i$ in $\{1,\ldots,n\}$. As both $z_{i}$ and $j(z_{i})$ are algebraic over $\mathbb{Q}$ we can apply Lemma \ref{lem:bound-degree} and conclude  that  $[\mathbb{Q}(z_{i},j(z_i)):\mathbb{Q}(z_{i})]$ is bounded above by a number $M$ depending only on $p_{1},\ldots,p_{n}$. By Lemma \ref{lem:class_number} there is a finite subset $S_M$ of $\mathbb{H}$ of special points, depending only on $M$, such that $z_i\in \Gamma \cdot S_M$ for every $i$ in $\{1,\ldots,n\}$. 
\end{proof}

\begin{cor}\label{cor:special_in_triangular_form}
Let $V\subseteq \mathbb{C}^{2n}$ be an  algebraic variety of dimension $n$ of triangular form. Assume that $V$  is not contained in any subvariety of the form $\mathbb{C}^n\times \{j(\mathbf{z})\}$ with $\mathbf{z}$ in $\mathbb{H}^n$ special. Then, $V_0$ contains infinitely many points of the form $\left(\mathbf{z}_0,j(\mathbf{z}_0)\right)$ with  $\mathbf{z}_0$ in $\mathbb{H}^{n}$ not special. \end{cor}
\begin{proof}
By Proposition \ref{prop:special}, the set of special points $\mathbf{z}$ in $\mathbb{H}^n$ with $(\mathbf{z},j(\mathbf{z}))\in V_0$ is contained in $(\Gamma\cdot S)^n$ for some finite set $S\subset \mathbb{H}$ of special points.  Put $J:=j(S)^n$. By the last statement in Proposition \ref{prop:conj1fortriangularform} the set
$E_j^n\cap V_0\cap (\mathbb{C}^{2n}\setminus (\mathbb{C}^n\times J))$
is infinite. This proves the desired result since every point in this set has at least one non-special coordinate.
\end{proof}

When working with varieties of triangular form defined over $\overline{\mathbb{Q}}$ and assuming MSC, we can prove stronger versions of Proposition \ref{prop:special} and Corollary \ref{cor:special_in_triangular_form}. 

\begin{prop}
\label{prop:mscspecial}
Let $V\subseteq\mathbb{C}^{2n}$ be a variety of triangular form of dimension $n$ defined over $\overline{\mathbb{Q}}$. Then MSC implies that there is a finite set $S\subset\mathbb{H}$ of special points, such that for every $\mathbf{z}$ in $\mathbb{H}^{n}$ with $\left(\mathbf{z},j(\mathbf{z})\right)\in V_{0}$, we have that the special coordinates of $\mathbf{z}$ are in $\Gamma\cdot S$.
\end{prop}
\begin{proof}
Let $p_{1},\ldots,p_{n}$ be polynomials defining $V$ satisfying the conditions of triangular form and let $C_0$ be the set of coefficients of the $p_{i}$'s. Suppose that $\mathbf{z}=(z_i)_{i=1}^n$ in $ \mathbb{H}^{n}$ has at least one special coordinate and satisfies $\left(\mathbf{z},j(\mathbf{z})\right)\in V_{0}$. By Proposition \ref{prop:special}, we can assume that $\mathbf{z}$ is not  special. Observe that by definition of $V_{0}$ we have
$$\mathrm{tr.deg.}_{\mathbb{Q}}(\mathbb{Q}(\mathbf{z},j(\mathbf{z}))) = \mathrm{tr.deg.}_{\mathbb{Q}}(\mathbb{Q}(\mathbf{z})).$$
As we are assuming MSC, we have
$$\mathrm{tr.deg.}_{\mathbb{Q}}(\mathbb{Q}(\mathbf{z},j(\mathbf{z})))\geq\dim_G(\{z_1,\ldots,z_n\}|\Sigma)$$
(see \textsection\ref{subsec:msc}). Points in $\mathbb{H}$  that are algebraically independent must be in different $G$-orbits and cannot be in the $G$-orbit of special points, as special points are algebraic. This implies that $\mathrm{tr.deg.}_{\mathbb{Q}}(\mathbb{Q}(\mathbf{z}))\leq \dim_G(\{z_1,\ldots,z_n\}|\Sigma)$. We conclude
\begin{equation}
\label{eq:msceq}
    \mathrm{tr.deg.}_{\mathbb{Q}}(\mathbb{Q}(\mathbf{z}))=\dim_G(\{z_1,\ldots,z_n\}|\Sigma).
\end{equation}
Put $A=\left\{i\in\left\{1,\ldots,n\right\}:z_{i}\text{ is special}\right\}$, and $B = \left\{1,\ldots,n\right\}\setminus A$. Note that $B\neq \emptyset$. Choose $T\subseteq B$ so that $\mathbf{z}_T:=(z_i)_{i\in T}$ satisfies
\begin{equation*}
    \mathrm{tr.deg.}_{\mathbb{Q}}(\mathbb{Q}(\mathbf{z}_{T}))=\mathrm{tr.deg.}_{\mathbb{Q}}(\mathbb{Q}(\mathbf{z})).
\end{equation*}
In particular, the set $\left\{z_{i}:i\in T\right\}$ is algebraically independent over $\mathbb{Q}$. By (\ref{eq:msceq}) 
for every $k$ in $B$, $z_k$ is in the $G$-orbit of a point in $\left\{z_{i}:i\in T\right\}$. This implies that
\begin{equation}
\label{eq:fieldgenerated}
    z_{k}\in\mathbb{Q}\left(\{z_{i}:i\in T\}\right) \text{ for every }k \text{ in } B.
\end{equation}
Let $\ell$ denote the number of elements of $A$. Put $C:=C_0\cup \{z_i:i \in A\cup T\}$ and $D:=d_1\cdots d_n$ where $d_i=\mathrm{deg}_{Y_i}(p_i)$. 
For every $i$ in $A$ we have
$$[\mathbb{Q}(C):\mathbb{Q}(z_{i})]_{\mathrm{alg}}\leq [\mathbb{Q}\left(C_0\cup\left\{z_{k}:k\in A\right\}\right):\mathbb{Q}] \leq 2^\ell[\mathbb{Q}\left(C_0\right):\mathbb{Q}].$$ 
By (\ref{eq:fieldgenerated}) and the fact that $(\mathbf{z},j(\mathbf{z}))\in V_0$ we also have $[\mathbb{Q}(C,j(z_i)):\mathbb{Q}(C)]\leq D$ for every $i$ in $\{1,\ldots,n\}$. As both $z_{i}$ and $j(z_{i})$ are algebraic over $\mathbb{Q}$ when $i$ belongs to $A$, we can apply Lemma \ref{lem:bound-degree} and conclude  that for every $i$ in $ A$, $[\mathbb{Q}(z_{i},j(z_i)):\mathbb{Q}(z_{i})]$ is bounded above by a number $M$ depending only on $p_{1},\ldots,p_{n}$. By Lemma \ref{lem:class_number} there is a finite subset $S_M$ of $\mathbb{H}$ of special points, depending only on $M$, such that $z_i\in \Gamma \cdot S_M$ for every $i$ in $\{1,\ldots,n\}$. This completes the proof.
\end{proof}

Using essentially the same proof as in Corollary \ref{cor:special_in_triangular_form} we conclude the following. 
\begin{cor}\label{cor:mscspecial_in_triangular_form}
Let $V\subseteq \mathbb{C}^{2n}$ be an  algebraic variety of dimension $n$ of triangular form defined over $\overline{\mathbb{Q}}$. Assume that $V$ is not contained in any subvariety of the form $\mathbb{C}^n\times \{j(\mathbf{z})\}$ with $\mathbf{z}$ in $\mathbb{H}^n$ non-ordinary. Then MSC implies that $V_0$ contains infinitely many points of the form $\left(\mathbf{z}_0,j(\mathbf{z}_0)\right)$ with  $\mathbf{z}_0$ in $\mathbb{H}^{n}$ ordinary. 
\end{cor}


\subsection{An unconditional case of Theorem \ref{th-main-2}}
\label{subsec:unconditional_result}
With the results we have so far, we can already prove a weaker but unconditional version of Theorem \ref{th-main-2}.
\begin{prop}
\label{prop:unconditional}
Let $V\subset\mathbb{C}^{2}$ be an irreducible curve defined over $\overline{\mathbb{Q}}$. Assume that $V$ is not a vertical line nor a horizontal line. Then  $V$ has infinitely many points of the form $(z,j(z))$ that are generic over $\overline{\mathbb{Q}}$. 
\end{prop}
\begin{proof}
By Schneider's theorem, it is enough to prove that $V$ has infinitely many points of the form $(z,j(z))$ with $z$ not special.  By Lemma \ref{lem:curves_are_triangular}, $V$ is of triangular form. Moreover, since $V$ is not a horizontal line, it satisfies the hypothesis of Corollary \ref{cor:special_in_triangular_form}, thus there are infinitely many points in $V_0$ of the form $(z,j(z))$ with $z$ not special. 
\end{proof}

\begin{remark}
Here is the analogous result for the exponential function. By \cite[Corollary 2.4]{marker} we know that an irreducible plane curve $V\subset \mathbb{C}^2$ that is not a vertical nor a horizontal line has infinitely many points of the form $(x,\exp(x))$ with $x\in \mathbb{C}$. So by Lindemann's theorem, we know that if $V$ is defined over $\overline{\mathbb{Q}}$, then for any non-zero $x$ in $\mathbb{C}$ with $(x,\exp(x))\in V$ we have $\mathrm{tr.deg.}_{\mathbb{Q}}(\mathbb{Q}(x,\exp(x))) = 1$.
\end{remark}

\begin{remark}\label{rem:unconditionalconsequence}
As a consequence of Proposition \ref{prop:unconditional} we have that there are infinitely many $z$ in $\mathbb{H}$ such that $j(z)=z$ which are transcendental over $\mathbb{Q}$ (we simply apply the proposition to the curve $V=\{(x,y)\in \mathbb{C}^2:x=y\}$). Furthermore, these points have to be in different $G$-orbits. Indeed, suppose that $z_{1},z_{2}$ in $\mathbb{H}$ are transcendental and satisfy $z_{1} = j(z_{1})$ and $z_{2}=j(z_{2})$. Suppose $g$ in $G$ is such that $gz_{1} = z_{2}$ and put $N:=\det(\mathrm{red}(g))$. Then the curves $aX+b=(cX+d)Y$, where $a,b,c,d$ are the entries of $g$ in the usual way, and $\Phi_{N}(X,Y)=0$ have the point $(j(z_{1}),j(z_{2}))$ in common. But, as both curves are defined over $\mathbb{Q}$, this point is generic in both curves. As both curves are irreducible, this would imply that the curves are in fact equal by Lemma \ref{lem:generic}, hence $N=1$ and $z_1=z_2$. This proves that different transcendental points $z$ in $\mathbb{H}$ with $j(z)=z$ are in different $G$-orbits. If we assume MSC, then we get that these points are algebraically independent.
\end{remark}

\subsection{Points with coordinates in fixed \texorpdfstring{$G$}{G}-orbits}\label{subs:G_orbit_bounded}

Here we prove that the coordinates in fixed $G$-orbits of points lying in $V_0$, where $V$ is a given algebraic variety of triangular form,  are somehow bounded in their $G$-orbit. This is the content of Proposition \ref{prop:gsol} below, but first we need some  preliminary results and notation. The following lemma is due to Pila (see \cite[Lemma 7.3]{pila}).

\begin{lem}\label{l:Pila}
Let $K$ be a finitely generated subfield of $\mathbb{C}$. Then there exist positive constants $c,\delta$ such that for every positive integer $N$ and every point $(x,y)$ in $\mathbb{C}^2$ with non-algebraic coordinates and satisfying $\Phi_N(x,y)=0$, we have that
$$[K(x,y):K]\geq c N^{\delta}.$$
\end{lem}

Given  elliptic curves $E,E'$ defined over a subfield of $\mathbb{C}$, we say that $E$ and $E'$ are isogenous if there exists an isogeny $\psi:E\to E'$ defined over $\mathbb{C}$. The following lemma follows from \cite[Lemmas 6.1 and 6.2]{masser-wustholz}.

\begin{lem}\label{l:masser-wustholz}
Let $K$ be a subfield of $\mathbb{C}$ and let $E,E'$ be elliptic curves defined over $K$ such that $E$ and $E'$ are isogenous. Then, the following properties hold.
\begin{enumerate}
    \item If $\psi:E\to E'$ is an isogeny, then $\psi$ is defined over a finite extension of $K$ contained in $\mathbb{C}$.
    \item If $\psi:E\to E'$ is an isogeny of minimal degree, then $\psi$ is cyclic.
\end{enumerate}
\end{lem}

The following theorem is a direct consequence of a result due to Pellarin (see \cite[Theorem 1]{pellarin}) whose proof makes use of Baker's method and explicit minorations of linear forms in logarithms.

\begin{thm}\label{t:Pellarin}
Let $K$ be a number field and let $E$ be an elliptic curve defined over $K$. Then, there is a positive constant $c$ depending only on $E$ such that for every elliptic curve $E'$ defined over $K$ and isogenous to $E'$ there exists an isogeny $\psi:E\to E'$ with
\begin{equation}\label{eq:t-Pellarin}
   \mathrm{deg}(\psi)\leq c[K:\mathbb{Q}]^5. 
\end{equation}
\end{thm}

\begin{remark}
Theorem \ref{t:Pellarin} remains valid if we replace the exponent $5$ in \eqref{eq:t-Pellarin} by any real number $\kappa$ with $\kappa>4$. We have chosen $\kappa=5$ for simplicity.
\end{remark}



Recall from \textsection\ref{sec:modular_polynomials} that for a matrix $g$ in $ G$, $\mathrm{red}(g)$ denotes the matrix obtained by re-scaling $g$ so that all the entries of $\mathrm{red}(g)$ are integers and relatively prime. Let $x,y$ in $\mathbb{H}$ and $g$ in $ G$ be such that $gx = y$. In this case, we denote by $g_{x,y}$ any element of $M_2^+(\mathbb{Z})$ satisfying 
\begin{equation*}
    \det(g_{x,y}) = \min\left\{\det(\mathrm{red}(g)) :g\in G, gx=y\right\}.
\end{equation*}
Note that, if $x$ is not special, then any other $h$ in $G$ satisfying $hx=y$ is of the form $rg_{x,y}$ for some non-zero rational number $r$. Hence, we have $\det(g_{x,y})=\det(\mathrm{red}(h))$ if $x$ is not special. 

We now state and prove the main result of this subsection.

\begin{prop}
\label{prop:gsol}
Let $V\subseteq\mathbb{C}^{2n}$ be a variety of triangular form  of dimension $n$. Then for every $\mathbf{z}$ in $\mathbb{H}^{n}$ there is a positive integer $M_0$ such that for every $g_{1},\ldots,g_{n}$ in $ G$ satisfying
$$(g_{1}z_{1},\ldots,g_{n}z_{n},j(g_{1}z_{1}),\ldots,j(g_{n}z_{n}))\in V_0,$$ 
we have $\det\left(g_{z_{i},g_{i}z_{i}}\right)<M_0$ for every $i$ in $ \{1,\ldots,n\}$.
\end{prop}
\begin{proof}
Let $p_{1},\ldots,p_{n}$ be polynomials defining $V$ and satisfying the conditions of triangular form. For $i$ in $\{1,\ldots,n\}$ let $d_i$ and $q_{i,k}$, for $k$ in $ \{1,\ldots,d_i\}$, be given by \eqref{eq:eqtriangularform}.  Let $C_0$ be the set of coefficients of the $p_{i}$'s. Fix $\mathbf{z}=(z_i)_{i=1}^n$ in $\mathbb{H}^n$ such that $(\mathbf{z},j(\mathbf{z}))\in V_{0}$, put $C:=C_0\cup \{z_1,\ldots,z_n\}$ and define $D:=d_1\cdots d_{n}$. Let $g_{1},\ldots,g_{n}$ be elements of $ G$. 
For every $i$ in $\left\{1,\ldots,n\right\}$ we have that $g_{i}z_{i}\in\mathbb{Q}(z_{i})$. If $g_{1},\ldots,g_{n}$ are such that $(g_{1}z_{1},\ldots,g_{n}z_{n},j(g_{1}z_{1}),\ldots,j(g_{n}z_{n}))\in V_0$, then 
$j(g_{1}z_{1})$ is algebraic over $\mathbb{Q}(C)$ and it generates a field extension whose degree is bounded above by $d_1$. Similarly, for $i$ in $\{2,\ldots,n\}$, 
$j(g_{i}z_{i})$ is algebraic over $\mathbb{Q}(C,j(g_1(z_1)),\ldots,j(g_{i-1}z_{i-1}))$ and it generates a field extension whose degree is bounded above by $d_{i}$. This implies that for each $i$ in $\{1,\ldots,2n-d\}$ we have $[\mathbb{Q}(C,j(g_{i}z_{i})):\mathbb{Q}(C)]\leq D$. By the equivalence between (M1) and (M2) in \textsection\ref{sec:modular_polynomials}, and  Lemma \ref{lem:bound-degree}, the degree $[\mathbb{Q}(j(z_{i}),j(g_{i}z_{i})):\mathbb{Q}(j(z_{i}))]$ is bounded above by a  constant depending only on $p_{1},\ldots,p_{n}$ and $\mathbf{z}$.\\
Given $z$ in $\mathbb{H}$ and  $g$ in $ G$, let $N$ be the smallest positive integer such that $\Phi_{N}(j(z),j(gz))=0$. We have   $N=\det\left(g_{z,gz}\right)$. 
We will now show that the degree $[\mathbb{Q}(j(z),j(gz)):\mathbb{Q}(j(z))]$ is bounded below by a number depending on $N$ and $z$. 
Observe first that $j(z)\in\overline{\mathbb{Q}}$ if and only if $j(gz)\in\overline{\mathbb{Q}}$. \\
We first consider the case $j(z)\notin\overline{\mathbb{Q}}$. Using Lemma \ref{l:Pila} we get that there are two positive constants $c$ and $\delta$ depending only on the field $\mathbb{Q}(j(z))$ such that
\begin{equation*}
    [\mathbb{Q}(j(z),j(gz)):\mathbb{Q}(j(z))]\geq cN^{\delta}.
\end{equation*}
We now consider the case $j(z)\in\overline{\mathbb{Q}}$. We then have $j(gz)\in\overline{\mathbb{Q}}$. Let $E_{1}$ and $E_{2}$ be two elliptic curves defined over the number field $K = \mathbb{Q}(j(z),j(gz))$ such that $j(E_{1}) = j(z)$ and $j(E_{2}) = j(gz)$ (see, e.g.~\cite[Chapter III, Proposition 1.4(c)]{silverman1}). By \cite[Chapter 5, \textsection 3, Theorem 5]{lang} we know that  there is an isogeny $\lambda: E_{1}\rightarrow E_{2}$ with cyclic kernel of degree $N$. Define $N'$ as the smallest positive integer such that there exists an isogeny $\psi:E_{1}\rightarrow E_{2}$ of degree $N'$. By Theorem \ref{t:Pellarin} there is a positive constant $c$ depending only on $E_{1}$ such that
    \begin{equation*}
        N'\leq c[K:\mathbb{Q}]^{5}.
    \end{equation*}
On the other hand, by Lemma \ref{l:masser-wustholz}(2) we have that $\psi$ is cyclic. This implies that $N=N'$ and $c_0N^{\frac{1}{5}}\leq [\mathbb{Q}(j(z),j(gz)):\mathbb{Q}(j(z))]$ where $c_0:=(c^{1/5}[\mathbb{Q}(j(z)):\mathbb{Q}])^{-1}$. This proves the desired lower bound for $[\mathbb{Q}(j(z),j(gz)):\mathbb{Q}(j(z))]$. \\
Since  $[\mathbb{Q}(j(z_{i}),j(g_{i}z_{i})):\mathbb{Q}(j(z_{i}))]$ is bounded above by a  constant depending only on $p_{1},\ldots,p_{n}$ and $\mathbf{z}$ we conclude the same for $\det(g_{z,g_{i}z})$. This proves the proposition.
\end{proof}

\begin{remark}
Obtaining results like Proposition \ref{prop:gsol} for the exponential function has proven to be a rather difficult problem, which has only been fully solved for the case of plane irreducible curves in the following sense: given an irreducible polynomial $p(X,Y)$ in $\mathbb{C}[X,Y]$ and a finitely generated subfield $K$ of $\mathbb{C}$, it has been proven that there exists a finite dimensional $\mathbb{Q}$-vector subspace $L$ of $\mathbb{C}$ such that every $x$ in $\overline{K}$ with $p(x,\exp(x))=0$ is contained in $L$. Moreover, given a basis $B$ of $L$ as a $\mathbb{Q}$-vector space, there exists a positive integer $N$ such that every such $x$ is a linear combination of elements in $B$ with coefficients in $\mathbb{Z}\left[\frac{1}{N}\right]$ (see \cite[Theorem 1.1]{gunaydin} and \cite[\textsection 2]{mantova}). 
\end{remark}

\subsection{Finishing the proof}\label{subsec:end_proof_th2}

We will complete the proof of Theorem \ref{th-main-2} after one last technical lemma, which is the only conditional result used in our proof.

\begin{lem}
\label{lem:findim}
Let $V\subseteq\mathbb{C}^{2n}$ be a variety of triangular form of dimension $n$, let $K$ be a finitely generated subfield of $\mathbb{C}$ and let $B$ the set of non-special coordinates of  points $\mathbf{z}$ in $\left(\mathbb{H}\cap\overline{K}\right)^{n}$ such that $\left(\mathbf{z},j(\mathbf{z})\right)\in V_0$. Then MSC implies that $\dim_G(B)$ is finite.
\end{lem}
\begin{proof}
Let $b_{1},\ldots,b_{m}$ be elements of $B$. For every $i$ in $\left\{1,\ldots,m\right\}$, choose a point $\mathbf{z}_{i}=(z_{ij})_{j=1}^n$ in $\left(\mathbb{H}\cap\overline{K}\right)^{n}$ such that  $\left(\mathbf{z},j(\mathbf{z})\right)\in V_0$ and at least one of the coordinates of $\mathbf{z}_{i}$ equals $b_{i}$. Let $F$ be a finitely generated field over which $V$ is defined. By hypothesis and the definition of $V_{0}$ we have 
\begin{eqnarray*}
    \mathrm{tr.deg.}_{\mathbb{Q}}(\mathbb{Q}(\mathbf{z}_{1},\ldots,\mathbf{z}_{m},j(\mathbf{z}_{1}),\ldots,j(\mathbf{z}_{m})))&\leq& \mathrm{tr.deg.}_{\mathbb{Q}}(F(\mathbf{z}_{1},\ldots,\mathbf{z}_{m},j(\mathbf{z}_{1}),\ldots,j(\mathbf{z}_{m}))) \\
    &=& \mathrm{tr.deg.}_{\mathbb{Q}}(F(\mathbf{z}_{1},\ldots,\mathbf{z}_{m}))\\ 
    &\leq& \mathrm{tr.deg.}_{\mathbb{Q}}(\mathbb{Q}(\mathbf{z}_{1},\ldots,\mathbf{z}_{m})) + \mathrm{tr.deg.}_{\mathbb{Q}}(F).
\end{eqnarray*}
Now, $\mathrm{tr.deg.}_{\mathbb{Q}}(\mathbb{Q}(\mathbf{z}_{1},\ldots,\mathbf{z}_{m}))\leq \mathrm{tr.deg.}_{\mathbb{Q}}\left(\overline{K}\right)$, which is finite. On the other hand, under MSC,
\begin{eqnarray*}
    \mathrm{tr.deg.}_{\mathbb{Q}}(\mathbb{Q}(\mathbf{z}_{1},\ldots,\mathbf{z}_{m},j(\mathbf{z}_{1}),\ldots,j(\mathbf{z}_{m}))) & \geq  & \dim_G(\{z_{ij}:i\in \{1,\ldots,m\},j\in\{1,\ldots,n\}\}|\Sigma) \\
    & \geq & \dim_G(\{b_{1},\ldots,b_{m}\}),
\end{eqnarray*}
hence $\dim_G(B)\leq\mathrm{tr.deg.}_{\mathbb{Q}}(F) + \mathrm{tr.deg.}_{\mathbb{Q}}\left(\overline{K}\right)$. This completes the proof.
\end{proof}

\begin{remark}
Lemma \ref{lem:findim} can be seen as an analogue of \cite[Proposition 2.2]{mantova}.
\end{remark} 

We can now give the proof of Theorem \ref{th-main-2}.

\begin{proof}[Proof of Theorem \ref{th-main-2}]
By Lemma \ref{lem:curves_are_triangular}, $V$ is of triangular form. By Lemma \ref{lem:findim} we know that there is a finite subset $B_0$ of $\mathbb{H}\cap \overline{K}$ such that, if $z$ in $\mathbb{H}\cap\overline{K}$ is not special and satisfies $(z,j(z))\in V_0$, then $z\in G\cdot B_0$. By Proposition \ref{prop:gsol}, there is a positive integer $M$ such that for every $z$ in $ B_0$ and every $g$ in $ G$ we have
$$(gz,j(gz))\in V_0\Rightarrow \mathrm{det}(\mathrm{red}(g))\leq M.$$
The group $\Gamma$ acts by left multiplication on the set $A=\{g\in \mathrm{M}^+_2(\mathbb{Z}) : \mathrm{det}(g)\leq M\}$ decomposing $A$ into finitely many $\Gamma$-orbits. This implies that the set
$$J_1:=\{j(g(z)):z\in  B_0,g\in G,(gz,j(gz))\in V_0\}$$
is finite. On the other hand, by Proposition \ref{prop:special}, the set
$$J_2:=\{j(z):z\in \mathbb{H} \mbox{ special},(z,j(z))\in V_0\}$$
is also finite. Put $J:=J_1\cup J_2$. Since $V$ is not contained in $\mathbb{C}\times J$, we can apply Proposition \ref{prop:conj1fortriangularform}. We obtain that the set
$$E_j^1\cap V_0\cap (\mathbb{C}^2\setminus (\mathbb{C}\times J))$$
is infinite. Every point in this set is of the form $(z,j(z))$ with $z$ in $\mathbb{H}$ not special and not in $\overline{K}$, hence satisfying $\mathrm{tr.deg.}_{K}(K(z,j(z))) = 1$. This completes our proof.
\end{proof}

\section{Proof of Theorem \ref{th:main3}}\label{sec:proof_of_thm_3}

In this section we present the proof of Theorem \ref{th:main3}. The main ingredient is our Proposition \ref{prop:masser}.

\begin{proof}[Proof of Theorem \ref{th:main3}]
We can write 
$$p(X,Y_1,\ldots,Y_n)=\sum_{i=0}^dp_i(X,Y_1,\ldots,Y_{n-1})Y_n^i$$
where $d$ is a positive integer and  $p_0,\ldots,p_d$ are polynomials in $\mathbb{C}[X,Y_1,\ldots,Y_{n-1}]$ with $p_d\neq 0$.
Let $R$ denote the ring $\mathbb{C}[X,Y_1,\ldots,Y_{n-1}]$ and let $r$ in $R$ be  the resultant of $p$ and $\frac{\partial p}{\partial Y_n}$ as polynomials in $R[Y_n]$ (\cite[Chapter 2,  \S2]{griffiths}). Since $p$ is irreducible, we have $r\neq 0$. There exist polynomials $F,G$ in $ R[Y_n]$ such that $Fp+G\frac{\partial p}{\partial Y_n}=r$. Since the product $rp_d$ is a non-zero polynomial in $R$, we can find a point $\mathbf{x}=(x,y_1,\ldots,y_{n-1})$ in $\mathbb{R}^n$ with $r(\mathbf{x})p_d(\mathbf{x})\neq 0$. Since $p_d(\mathbf{x})\neq 0$ and $\mathbb{C}$ is algebraically closed, we can find $y_n$ in $\mathbb{C}$ such that $p(\mathbf{x},y_n)=0$. Since $r(\mathbf{x})\neq 0$, we have $\frac{\partial p}{\partial Y_n}(\mathbf{x},y_n)\neq 0$. By the Implicit Function Theorem, there exists neighbourhoods $U_1, U_2$  of $\mathbf{x}$ in $\mathbb{C}^n$ and of $y_n$ in $\mathbb{C}$, respectively, and a holomorphic function $H:U_1\to U_2$ such that
$$\{\mathbf{z}\in U_1\times U_2:p(\mathbf{z})=0\}=\{(\mathbf{w},H(\mathbf{w})):\mathbf{w}\in U_1\}.$$
Since $U_1\cap \mathbb{R}^n\neq \emptyset$, we can apply Proposition \ref{prop:masser} to the system of equations
\begin{eqnarray*}
j(z)& =& z_1,\\
j(z_1)&=& z_2,\\
 & \vdots & \\
 j(z_{n-1}) &=& H(z,z_1,\ldots,z_{n-1}),
\end{eqnarray*}
for $(z,z_1,\ldots,z_{n-1})$ in $ U_1\cap \mathbb{H}^{n}$. Hence, this system has infinitely many solutions. Different solutions of this system give different complex numbers $z$ in $\mathbb{H}_n$ with $(z,j(z),\ldots,j_n(z))\in V$. This completes the proof of the theorem.
\end{proof}

\section{Proof of Theorem \ref{thm:main4}}\label{sec:proof_of_thm_4}

The proof of Theorem \ref{thm:main4} is given at the end of this section after a couple of lemmas. \\
Given a matrix $g=\left(\begin{array}{cc}
    a & b \\
    c & d
\end{array}\right)$ in $G$ we define the rational function $gX:=\frac{aX+b}{cX+d}$. 

\begin{lem}\label{lem:G-divisors}
Let $p(X,Y)$ be a non zero polynomial in $\mathbb{C}[X,Y]$ and let $m$ denote the degree of $p(X,Y)$ with respect to the $X$ variable. Define
$$D(p;G)=\left\{g=\left(\begin{array}{cc}
    a & b \\
    c & d
\end{array}\right)\in M^+_2(\mathbb{Z}):\mathrm{gcd}(a,b,c,d)=1, (cX+d)^mp(gX,X)=0\right\}.$$
Then, $D(p;G)$ is finite.
\end{lem}
\begin{proof}
Let $g=\left(\begin{array}{cc}
    a & b \\
    c & d
\end{array}\right)$ be a matrix in $M^+_2(\mathbb{Z})$. We have $(cX+d)^mp(gX,X)=0$  if and only if $p$ vanishes on $V(L_g)$, where $L_g(X,Y):=dY-b+X(cY-a)$. Since $L_g$ is irreducible in $\mathbb{C}[X,Y]$, this means that $L_g$ divides $p$. For $g,g'$ in $ D(p;G)$ with $g\neq \pm g'$ we have that the irreducible polynomials $L_g$ and $L_{g'}$ are not associated. Since $p$ is divisible by only finitely many non-associated irreducible polynomials, we conclude that $D(p;G)$ is finite. This proves the lemma. 
\end{proof}

\begin{lem}\label{lem:finite-Js}
Let $p(X,Y_1,Y_2)$ be a non zero polynomial in $\overline{\mathbb{Q}}[X,Y_1,Y_2]$ with $\frac{\partial p}{\partial Y_2}\neq 0$. Write
$$p(X,Y_1,Y_2)=\sum_{i=0}^dp_i(X,Y_1)Y_2^i,$$
where $p_i(X,Y_1)\in \overline{\mathbb{Q}}[X,Y_1]$ for $i$ in $\{0,1,\ldots,d\}$ and $p_d\neq 0$, and let
$$W=\{(x,y_1,y_2):p(x,y_{1},y_{2})=0, p_d(x,y_1)\neq 0\}.$$
Then there exist finite sets $J_1,J_2\subset \mathbb{C}$ such that
$$j\left(\{z\in \mathbb{H}_2:z\text{ is special and } (z,j(z),j_2(z))\in W\}\right)\subseteq  J_1$$
and
$$j_2\left(\{z\in \mathbb{H}_2:z\text{ is not special}, j(z)\text{ is special and } (z,j(z),j_2(z))\in W\}\right) \subseteq J_2.$$
\end{lem}
\begin{proof}
We claim that there exist only finitely many $\Gamma$-orbits of special points $z$ in $\mathbb{H}_2$ with $(z,j(z),j_2(z))\in W$. Indeed, let $C_0$ as the set of coefficients of $p$. If $z$ is a special point in $\mathbb{H}_2$ with $(z,j(z),j_2(z))\in W$, then $j_2(z)$ is algebraic over $\mathbb{Q}(C_0,z,j(z))$, 
hence it is algebraic over $\mathbb{Q}$. This implies that $j(z)$ is a special point in $\mathbb{H}$, therefore $[\mathbb{Q}(z,j(z)):\mathbb{Q}(z)]\leq 4$. By Lemma \ref{lem:class_number} we have that $z$ must belong to a finite set of $\Gamma$-orbits that is independent of $z$, as claimed. We conclude that there exists a finite set $J_1\subset \mathbb{C}$ such that for every special point $z$ in $\mathbb{H}_2$ with $(z,j(z),j_2(z))\in W$ we have $j(z)\in J_1$.\\
%
Now, assume that $z$ is a non-special point in $\mathbb{H}_2$ with $j(z)$ special and $(z,j(z),j_2(z))\in W$. Since $j(z)$ is algebraic and $z$ is not special, we have that $z$ is transcendental. Put $C:=C_0\cup \{z\}$. We have that $j_2(z)$ is algebraic over $\mathbb{Q}(j(z))$, $$[\mathbb{Q}(C,j(z),j_2(z)):\mathbb{Q}(C,j(z))]=[\mathbb{Q}(C_0,z,j(z),j_2(z)):\mathbb{Q}(C_0,z,j(z))]\leq d$$
and 
\begin{eqnarray*}
  [\mathbb{Q}(C,j(z)):\mathbb{Q}(j(z))]_{\text{alg}} &=&[\mathbb{Q}(C_0,z,j(z)):\mathbb{Q}(j(z))]_{\text{alg}}\\
  &\leq &[\mathbb{Q}(C_0,z,j(z)):\mathbb{Q}(j(z),z)]\\
  &\leq& [\mathbb{Q}(C_0):\mathbb{Q}].
\end{eqnarray*}
By Lemma \ref{lem:bound-degree} we conclude that $[\mathbb{Q}(j_2(z),j(z)):\mathbb{Q}(j(z))]$ is bounded above by a constant depending only on $p$. Then, by Lemma \ref{lem:class_number}, we have that $j(z)$ must belong to a finite set of $\Gamma$-orbits that depends only on $p$. This proves that there exists a finite set $J_2\subset \mathbb{C}$ depending only on $p$ such that for every non-special point $z$ in $\mathbb{H}_2$ with $j(z)$ special and $(z,j(z),j_2(z))\in W$ we have $j_2(z)\in J_2$. This completes the proof of the lemma.
\end{proof}

\begin{proof}[Proof of Theorem \ref{thm:main4}]
There exists an integer $d\geq 1$ and polynomials $p_0,\ldots,p_d$ in $\overline{\mathbb{Q}}[X,Y_1]$ with $p_d\neq 0$ such that
$$p(X,Y_1,Y_2)=\sum_{i=0}^dp_i(X,Y_1)Y_2^i.$$
Put $W=\{(x,y_1,y_2)\in V: p_d(x,y_1)\neq 0\}$ and let $J_1,J_2$ be the finite subsets of $\mathbb{C}$ given by Lemma \ref{lem:finite-Js}. Define $r$ in $\mathbb{C}[X,Y_1]$ as the resultant of the polynomials $p$ and $\frac{\partial p}{\partial Y_2}$ (\cite[Chapter 2,  \S2]{griffiths}). The set 
\begin{eqnarray*}
A & := &\{(x,y_1)\in \mathbb{C}^2:r(x,y_1)p_d(x,y_1)\neq 0, y_1\not \in J_1\}
\end{eqnarray*}
is Zariski open in $\mathbb{C}^2$ and non-empty. 
Thus, we can find a point $(x_0,y_0)$ in $A\cap \mathbb{R}^2$. We choose a point $y_1$ in $\mathbb{C}$ with $p(x_0,y_0,y_1)=0$ and use the Implicit Function Theorem in order to find open neighbourhoods $U_1$ of $(x_0,y_0)$ in $\mathbb{C}$ and  $U_2$ of  $y_1$ in $\mathbb{C}$ such that
$$V\cap (U_1\times U_2)=\{(z,z_1,H(z,z_1)):(z,z_1)\in U_1\}$$
for some holomorphic function $H:U_1\to U_2$. By shrinking $U_1$, if necessary, we can assume that $U_1$ is connected and $U_1\subseteq A$. For a positive integer $N$ (to be chosen conveniently later) we define $F_N$ in $\mathbb{C}[Y_1,Y_2]$ as 
$$F_N(Y_1,Y_2)=\prod_{i=1}^N\Phi_i(Y_1,Y_2)\cdot \prod_{t \in J_2}(Y_2-t),$$
where $\Phi_1,\ldots,\Phi_N$ are the first $N$ modular polynomials. We claim that
\begin{equation}\label{eq:F-non-zero}
    \{(z,z_1)\in U_1\cap \mathbb{R}^2:F_N(z_1,H(z,z_1))\neq 0\}\neq \emptyset.
\end{equation}
Indeed, if this set were empty, then we would have $F_N(z_1,H(z,z_1))=0$ for every $(z,z_1)$ in $U_1$ and then
$$V\cap (U_1\times U_2)\subseteq V(p,F_N).$$
By Corollary \ref{cor:generic_contention} we get that $V$ is contained in $V(p,F_N)$ thus $p$ must divide $F_N$, which is impossible since $F_N$ depends only on $Y_1$ and $Y_2$ while $p$ depends also on $X$. This proves \eqref{eq:F-non-zero}.\\
Consider the system of equations
\begin{equation}\label{system}
\begin{array}{ccc}
j(z) &=& z_{1}, \\
j(z_{1}) &=& H(z,z_1), \\
\frac{1}{j(w)} &=& F_N\left(z_1,H(z,z_1)\right),
\end{array}
\end{equation}
for $(z,z_1,w)$ in $\mathbb{H}^{3}\cap (U_1\times \mathbb{C})$. Since $\frac{1}{j}$ attains all values in $\mathbb{C}\setminus\{0\}$ and $F_N(z_1,H(z,z_1))$ is non identically zero (by \eqref{eq:F-non-zero}) we can apply Proposition \ref{prop:masser}. This way we obtain infinitely many solutions of this system. \\
For the rest of the proof, we let $(z,z_1,w)$ be a solution of \eqref{system}. Clearly, we have $z\in \mathbb{H}_2$ and
$$(z,j(z),j_2(z))=(z,z_1,H(z,z_1))\in W\subseteq V.$$
Since $U_1\subseteq A$ we have $j(z)\not \in J_1$ thus $z$ is not special. Moreover, since $F\left(z_1,H(z,z_1)\right)=\frac{1}{j(w)}\neq 0$, we have $H(z,z_1)\not \in J_2$ and $j(z)$ is also not special.
Now, assuming MSC 
we get the inequality
$$\mathrm{tr.deg.}_{\mathbb{Q}}(\mathbb{Q}(z,j(z),j(z),j_{2}(z)))\geq \mathrm{dim}_G(z,j(z)).$$
In order to complete our proof, we will show that $z$ and $j(z)$ cannot be in the same $G$-orbit if $N$ is big enough. Indeed, let $m$ be the degree of $p(X,Y_{1},Y_{2})$ in the $X$ variable, let $D(p_d;G)$ be the set defined in Lemma \ref{lem:G-divisors} and put $N_1=\max\{\det(g):g\in D(p_d;G)\}$ (choose $N_1=1$ if $D(p_d;G)$ is empty). Define $N_2$ as the total degree of $p$. We choose $N$ to be the any integer with $N\geq \max\{N_1,N_2\}$. Now, assume that $z$ and $j(z)$ are in the same $G$-orbit. 
Let $g$ be an element of  $G$ such that $z = g j(z)$. Using the notation of \textsection\ref{sec:modular_polynomials}, let $a,b,c,d$ be the entries (in the usual way) of $\mathrm{red}(g)$. Put $M=\mathrm{det}(\mathrm{red}(g))$. Then $\Phi_{M}(j(z),j_{2}(z)) =0$. Define the polynomial 
\begin{equation*}
    q(X,Y) := (cX+d)^{m}p\left(\mathrm{red}(g)X,X,Y\right) \text{ in }\overline{\mathbb{Q}}[X,Y],
\end{equation*}
and set $V_{1} = V(q)$ and $V_{2}=V(\Phi_M)$ as affine subvarieties of $\mathbb{C}^{2}$. Observe that $(j(z),j_2(z))\in V_{1}\cap V_{2}$. As $j(z)$ is not special, it cannot happen that both $j(z)$ and $j_2(z)$ are algebraic. Since $V_{2}$ is irreducible and defined over $\mathbb{Q}$ (hence over $\overline{\mathbb{Q}}$), and $\text{tr.deg.}_{\mathbb{Q}}(\mathbb{Q}(j(z),j_2(z)))=1$, we get $V_{2} \subseteq V_{1}$ by Lemma \ref{lem:generic}. This means that either $V_{1} = \mathbb{C}^{2}$ or $\Phi_{M}(X,Y)$ divides $q(X,Y)$. In the first case we have $q(X,Y)=0$, thus $(cX+d)^{m}p_d(\mathrm{red}(g)X,X)=0$ and $\mathrm{red}(g)\in D(p_d;G)$. This implies that $M\leq N$ hence $F_N(j(z),j_2(z))=0$, which is a contradiction since $F_N(j(z),j_2(z))=\frac{1}{j(w)}\neq 0$. In the second case $\Phi_{M}(X,Y)$ divides $q(X,Y)$, thus
$$M\leq \deg(\Phi_M)\leq \deg q \leq N_2\leq N,$$
and $F_N(j(z),j_2(z))=0$, which gives the same contradiction as in the first case. This proves that $z$ and $j(z)$ cannot be in the same $G$-orbit and, by MSC, we conclude that $\mathrm{tr.deg.}_{\mathbb{Q}}(\mathbb{Q}(z,j(z),j_{2}(z)))=2$. By choosing different points $(x_0,y_0)$ in $A\cap \mathbb{R}^2$ we get that there are actually infinitely many $z$ in $\mathbb{H}_2$ satisfying the desired conditions. This completes the proof of the theorem.
\end{proof} 


\section{Further directions}
\label{sec:further}


The purpose of this final section is to show how the same ideas used in our proof of Proposition \ref{prop:masser} can be applied in other contexts. Specifically, we present two results: the existence of solutions of certain analytic equations involving $j'$ (the usual derivative of the $j$ function), and of equations involving the function $\exp(1/z)$. As mentioned in the introduction, we expect to study this type of problems in more detail in future work.


\begin{prop}
\label{prop:j'}
 Let $U$ be an open subset of $\mathbb{C}$ with $U\cap \mathbb{R}\neq \emptyset$ and  let $H:U\rightarrow\mathbb{C}$ be a holomorphic function. 
 Then there are infinitely many $z$ in $\mathbb{H}\cap U$ such that
\begin{eqnarray*}
j'(z) &=& H(z).
\end{eqnarray*}
\end{prop}
\begin{proof}
Let $\rho$ denote the complex number $\frac{-1}{2}+i\frac{\sqrt{3}}{2}$. It is known that $ j'(\rho)=0$. Since $j'$ satisfies the transformation property
$$j'\left(\gamma z\right)=(cz+d)^2j'(z) \text{ for every }z \text{ in } \mathbb{H} \text{ and every }\gamma=\left(\begin{array}{cc}a&b\\ c&d\end{array}\right)\text{ in } \Gamma,$$
we  have $j'(\gamma \rho) = 0$ for every $\gamma$ in $\Gamma$. Choose a small closed disk $B\subset\mathbb{H}$ around $\rho $ such that $j'(z)$ does not vanish on the boundary of $B$. Define $l:=\min \{\mathrm{Im}(z):z\in B\}$, put $\delta=\min\{|j'(z)|:z\in \partial B\}$ and let~$x_0$ be a point in~$U\cap \mathbb{R}$. Since $H$ is continuous, there exists an open neighbourhood  $V$ of $x_0$ contained in $U$ such that for every $z$ in $V$ we have $|H(z)|<|H(x_0)|+1$. Let $(\gamma_{n})_{n=1}^{\infty}$ be a sequence of elements in $\Gamma$ such that $|\gamma_{n}\rho - x_0|\to 0$ as $n$ goes to infinity. Then, we have $|\gamma_n z- x_0|\to 0$ as $n$ goes to infinity, uniformly for $z$ in $B$ (see the proof of Proposition \ref{eq:masser} in \textsection\ref{sec:equations_automorphic}), and there exists a positive integer $N_1$ such that $\gamma_n B\subseteq V$ for every $n\geq N_1$. \\
For every $\gamma=\left(\begin{array}{cc}a&b\\ c&d\end{array}\right)$ in  $\Gamma$ and every $z$ in $B$ we have
$$\mathrm{Im}(\gamma z)=\frac{\mathrm{Im}(z)}{|cz+d|^2}\geq \frac{l}{|cz+d|^2}.$$ 
Put $\gamma_n=\left(\begin{array}{cc}a_n&b_n\\ c_n&d_n\end{array}\right)$. Since $\mathrm{Im}(\gamma_n z)\to 0$  uniformly for $z$ in $B$, we have $|c_n z+d_n|^2\to \infty$ also uniformly for $z$ in $B$. It follows that  there exists a positive integer $N_2$ such that for every integer $n\geq N_2$ and every $z$ in $B$ we have $|c_nz+d_n|^2>\delta ^{-1}(|H(x_0)|+1)$. 
Define $N:=\max\{N_1,N_2\}$. Then, for every $z$ in $\partial B$ and every $n\geq N$ we have
$$|j'(\gamma_n z)|=|c_nz+d_n|^2|j'(z)|\geq |H(x_0)|+1>|H(z)|.$$
Hence, we can apply Rouch\'e's theorem to the functions $f(z):=j'(z)$ and $g(z):=-H(z)$ on $\gamma_nB$. Since $f$ has a zero in the interior of $\gamma_nB$ (namely $\gamma_n\rho$), we get that $(f+g)(z)=j'(z)-H(z)$ also has a zero there. In particular, the equation $j'(z)=H(z)$ has a solution in $\gamma_nB$. This implies the desired result.
\end{proof}

\begin{remark}
We point out that the proof of Proposition \ref{prop:j'} can be adapted to show that the equation $j''(z) = H(z)$ has infinitely many solutions. This is because the second derivative of $j$ also satisfies $j''(\gamma\rho) = 0$ for every $\gamma$ in $\Gamma$. 
\end{remark}

\begin{prop}
Let $U$ be an open subset of $\mathbb{C}$ containing $0$, and let $f$ be a holomorphic function defined on $U$ such that $f(0)\neq 0$. Then there are infinitely many $z$ in $U$ such that
\begin{equation*}
\label{eq:exp}
\exp(1/z) = f(z).    
\end{equation*}
\end{prop}
\begin{proof}
Let $\Lambda$ denote the group of matrices 
\begin{equation*}
\Lambda=\left\{\gamma_{n}:=
\left(\begin{matrix}
1 & 0 \\
2\pi in & 1
\end{matrix}\right): n\in\mathbb{Z}\right\}.
\end{equation*}
Then $\Lambda$ acts on $\widehat{\mathbb{C}}=\mathbb{C}^{\times}\cup \{\infty\}$ through M\"obius transformations and $\exp(1/z)$ is invariant under this action.  Put $\alpha = f(0)$. Let $x$ in $\mathbb{C}^{\times}$ be such that $\exp(1/x) = \alpha$. Let $B$ be a small closed disk around $x$ contained in $\mathbb{C}^{\times}$  such that $\exp(1/z)\neq\alpha$ for all $z$ in $\partial B$. Set $g(z):= \exp(1/z) - \alpha$ and $h(z):= \alpha - f(z)$. As $g(z)$ does not vanish on $\partial B$, the real number $\delta:=\min \left\{|g(z)|:z\in\partial B\right\}$ is positive. Furthermore, as $g$ is invariant under the action of $\Lambda$, we have $\delta=\min \{|g(z)|:z\in\partial (\gamma_{n}B)\}$ for every integer $n$.\\
Observe that for every $z$ in $\mathbb{C}^{\times}$, $\gamma_{n}z\to 0$ as $n$ goes to infinity. Moreover, for every compact subset $A$ of $\mathbb{C}^{\times}$ the convergence is uniform for $z$ in $A$. As $f$ is continuous at $0$ we have that $f(\gamma_{n}z)\to\alpha$ as $n$ goes to infinity, uniformly for $z$ in $B$. This implies that there exists a positive integer $N$ such that for every integer $n>N$, we have $|h(z)|<\delta$ for every $z$ in $\gamma_{n}B$. Hence, if $n>N$ we have $|g(z)|\geq \delta>|h(z)|$ for every $z$ in $\partial(\gamma_{n}B)$. By Rouch\'e's theorem, we get that $g(z)$ and $g(z)+h(z)$ have the same number of zeros (counting orders) in $\gamma_{n}B$, for all $n>N$. Now, $g(\gamma_{n}x) = g(x) = 0$ so $g(z)+h(z) = \exp(1/z) - f(z)$ has at least one zero in $\gamma_{n}B$. This implies the desired result and completes the proof of the proposition. 
\end{proof}


\section*{Acknowledgements}
The first author would like to thank Vahagn Aslanyan and Jonathan Pila for many useful discussions around the topics presented here. We  thank the anonymous referees for their valuable comments that helped us improve the exposition.

The first author was partially funded by CONICYT PFCHA/Doctorado Becas Chile/2015-72160240. The second author was partially funded by CONICYT FONDECYT/Postdoctorado Nacional 3190086.

\bibliographystyle{siam}

\end{document}